\documentclass[12pt]{amsart}

\usepackage{mathtools}
\newcommand{\norm}[1]{\left\lVert#1\right\rVert}
\newcommand{\abs}[1]{ \left\lvert#1\right\rvert}
\DeclarePairedDelimiter\ceil{\lceil}{\rceil}
\DeclarePairedDelimiter\floor{\lfloor}{\rfloor}

\usepackage{amsrefs}
\usepackage{amssymb}
\usepackage{fancyhdr}
\usepackage{bbm}
\usepackage[all]{xy}
\usepackage{xcolor}
\usepackage{hyperref}
\hypersetup{
  colorlinks=true,
  linkcolor=blue,
  citecolor=red
}

\newtheorem{theorem}{Theorem}
\newtheorem*{theorem*}{Theorem}
\newtheorem{lemma}[theorem]{Lemma}
\newtheorem{proposition}[theorem]{Proposition}
\newtheorem{claim}[theorem]{Claim}

\newtheorem{remark}[theorem]{Remark}

\newtheorem{notation}{Notation}

\theoremstyle{definition}
\newtheorem{definition}[theorem]{Definition}
\newtheorem*{definition*}{Definition}
\newtheorem*{lemma*}{Lemma}
\usepackage{graphicx}
\usepackage{pstricks, enumerate, pst-node, pst-text, pst-plot}

\numberwithin{equation}{section}
\numberwithin{theorem}{section}

\usepackage{xparse}
\DeclareDocumentCommand\Pr{ m g }{%
    {   \IfNoValueTF {#2}
      {\mathbb{P}\left[{#1}\right]}
      {\mathbb{P}\left[{#1}\middle\vert{#2}\right]}%
    }
}
\DeclareDocumentCommand\E{ m g }{%
    {   \IfNoValueTF {#2}
      {\mathbb{E}\left[{#1}\right]}
      {\mathbb{E}\left[{#1}\middle\vert{#2}\right]}%
    }
}

\newcommand{\R}{\mathbb{R}}
\newcommand{\N}{\mathbb{N}}

\newcommand{\Z}{\mathbb{Z}}

\newcommand{\eps}{\varepsilon}

\DeclareMathOperator{\PBA}{PBA}
\DeclareMathOperator{\DOM}{DOM}
\DeclareMathOperator{\END}{END}
\DeclareMathOperator{\STR}{STR}
\DeclareMathOperator{\irr}{irr}
\DeclareMathOperator{\domain}{domain}
\DeclareMathOperator{\diff}{diff}
\DeclareMathOperator{\troot}{root}
\DeclareMathOperator{\depth}{depth}
\DeclareMathOperator{\tstring}{String}
\DeclareMathOperator{\stree}{Tree}
\DeclareMathOperator{\IND}{ind}
\newcommand{\arr}[2]{\vec{#1}^{\;#2}}

\DeclareMathOperator{\supp}{supp}

\begin{document}

\title[]{Quasi-Regular Sequences}
\author[]{Joshua Frisch}
\author[]{Wade Hann-Caruthers}
\author[]{Pooya Vahidi Ferdowsi}
\address{California Institute of Technology}

\date{\today}

\begin{abstract}
Let $\Sigma$ be a countable alphabet. For $r\geq 1$, an infinite sequence $s$ with characters from $\Sigma$ is called $r$-quasi-regular, if for each $\sigma\in\Sigma$ the ratio of the longest to shortest interval between consecutive occurrences of $\sigma$ in $s$ is bounded by $r$. In this paper, we answer a question asked in ~\cite{kempe2018quasi}, and prove that for any probability distribution $\mathbf{p}$ on a finite alphabet $\Sigma$, there exists a $2$-quasi-regular infinite sequence with characters from $\Sigma$ and density of characters equal to $\mathbf{p}$. We also prove that as $\norm{\mathbf{p}}_\infty$ tends to zero, the infimum of $r$ for which $r$-quasi-regular sequences with density $\mathbf{p}$ exist, tends to one. This result has a corollary in the Pinwheel Problem: as the smallest integer in the vector tends to infinity, the density threshold for Pinwheel schedulability tends to one.
\end{abstract}

\maketitle

\section{Introduction}

Let $\Sigma$ be a countable alphabet, and $\mathbf{p}$ be a probability distribution on $\Sigma$. An infinite sequence $s\in\Sigma^{\N}$ has density $\mathbf{p}$, if for each $\sigma\in\Sigma$, the frequency of occurrences of $\sigma$ in the first $n$ characters of $s$ tends to $\mathbf{p}_\sigma$ as $n$ goes to infinity, i.e.\ if $\lim_{n\to\infty} \frac{1}{n}\abs{s^{-1}(\sigma)\cap[n]} = \mathbf{p}_\sigma$ for each $\sigma\in\Sigma$, where $[n] = \{1,\ldots,n\}$ for each $n\in\N$. If $s\in\Sigma^{\N}$ has density $\mathbf{p}$, we write $d_s=\mathbf{p}$. There are different notions of regularity for such sequences, and for each notion, a natural question to ask is how regular could sequences with density $\mathbf{p}$ be.

One notion of regularity is that of discrepancy. Discrepancy of a sequence $s\in\Sigma^{\N}$ is defined as $\sup_{\sigma\in\Sigma, n\in\N} \abs{\abs{s^{-1}(\sigma)\cap[n]} - n \mathbf{p}_\sigma}$.
This notion has been extensively studied ~\cite{tijdeman1973distribution, spencer1982sequences, angel2009discrete, tijdeman1980chairman} and a result by Tijdeman ~\cite{tijdeman1973distribution} shows that for any countable alphabet $\Sigma$ and for any probability distribution $\mathbf{p}$ on $\Sigma$, there is a sequence $s\in\Sigma^{\N}$ with density equal to $\mathbf{p}$ and discrepancy less than or equal to $1$. Later in this paper, we will use this result in some of our proofs.

Another notion is quasi-regularity. For $r\geq 1$, a sequence $s\in\Sigma^{\N}$ is called $r$-quasi-regular if for each $\sigma\in\Sigma$ the ratio of the longest to shortest interval between consecutive occurrences of $\sigma$ is bounded by $r$ ~\cite{kempe2018quasi}. The question is to find, for a $\Sigma$ and $\mathbf{p}$, the least $r$ with a $r$-quasi-regular sequence $s\in\Sigma^{\N}$ with density $\mathbf{p}$.

For example, if $\Sigma$ is finite and $\mathbf{p}$ is the uniform distribution on $\Sigma$, then the periodic sequence with distance between consecutive occurrences of each character equal to $\abs{\Sigma}$ gives us a $1$-quasi-regular sequence with density $\mathbf{p}$. As another example, if all the probabilities in $\mathbf{p}$ are powers of two, again, there is a $1$-quasi-regular sequence with density $\mathbf{p}$. Kempe, Schulman, and Tamuz in ~\cite{kempe2018quasi}*{Proposition 4.6} show that for $\mathbf{p}=(1/2,1/3,1/6)$ there are no $r$-quasi-regular sequences with density $\mathbf{p}$ for $r<2$. So, $r=2$ is the least possible value for which existence of $r$-quasi-regular sequences with arbitrary densities is guaranteed. Using an ergodic theoretical approach, they also show ~\cite{kempe2018quasi}*{Corollary 5.4} that for any finite $\Sigma$ and any probability distribution $\mathbf{p}$ on $\Sigma$, there always exists a $3$-quasi-regular sequence with density $\mathbf{p}$, and for $2\leq r <3$ they ask whether $r$-quasi-regular sequences always exist.

This paper contains two main results. The first one is a positive answer to the aforementioned question for finite alphabets; Theorem ~\ref{thm:2-quasi-regular} shows that for any finite alphabet $\Sigma$ and any probability distribution $\mathbf{p}$ on $\Sigma$, there exists a $2$-quasi-regular sequence $s\in\Sigma^{\N}$ with density $\mathbf{p}$. The second result studies existence of $r$-quasi-regular sequences when the probabilities in the distribution tend to $0$; Theorem ~\ref{thm:eps-quasi-regular} shows that as $\norm{\mathbf{p}}_\infty$ tends to zero, the infimum of $r$ for which $r$-quasi-regular sequences with density $\mathbf{p}$ exist, tends to one.

\subsection*{Pinwheel Problem} Given a vector $v = (v_1,\ldots,v_n)$ of integers with $v_1, \ldots, v_n\geq 2$, a Pinwheel schedule for $v$ is a sequence $s\in\{1,\ldots,n\}^{\N}$ such that for any $k\in\{1,\ldots,n\}$, there is at least one character $k$ in any consecutive subsequence of length $v_k$ in $s$ ~\cite{lin1997pinwheel, fishburn2002pinwheel, holte1989pinwheel, chan1993schedulers}. The density of $v$, denoted by $d(v)$, is equal to $\sum_{k=1}^{n} \frac{1}{v_k}$. Connections between existence of Pinwheel schedules for vector $v$ and density of $v$ have been studied: it is proven ~\cite{fishburn2002pinwheel} that any vector $v$ with $d(v)\leq 7/10$ has a Pinwheel schedule, and, furthermore, it is conjectured ~\cite{chan1993schedulers} that any vector $v$ with $d(v)\leq 5/6$ has a Pinwheel schedule. It is obvious that for $d(v) > 1$, no Pinwheel schedule for $v$ exists, and it is known that for any $\eps>0$, there exists a vector $v$ with $5/6 < d(v) <5/6 + \eps$ and no Pinwheel schedules.

An easy corollary of Theorem ~\ref{thm:eps-quasi-regular}, shows that for each $\eps>0$, there exists $N\in\N$ such that any vector $v=(v_1,\ldots,v_n)$ with $v_1,\ldots,v_n\geq N$ and $d(v) < 1-\eps$ has a Pinwheel schedule. This means that as the smallest integer in the vector tends to infinity, the density threshold for Pinwheel schedulability tends to one.

\subsection*{\texorpdfstring{$r$}{r}-Quasi-Regular Sequences}
Let $\Sigma$ be a countable alphabet. Define 
\begin{align*}
  \text{MaxGap} & : \Sigma^\N \times \Sigma \to \N \cup \{ \infty \},\\
  \text{MinGap} & : \Sigma^\N \times \Sigma \to \N \cup \{\infty \}
\end{align*}
as follows:
\begin{align*}
  \text{MaxGap}(s, \sigma) & = \sup{\{n \in \N \ | \ \exists k \in \N \ \text{s.t.} \ s_i \neq \sigma \ \text{for} \ i\in (k, k+n)\}},\\
  \text{MinGap}(s, \sigma) & = \inf{\{n \in \N \ | \ \exists k \in \N \ \text{s.t.} \  s_k = s_{k+n} = \sigma\}}.
\end{align*}

So, for $s\in \Sigma^\N$ and $\sigma\in \Sigma$, $\text{MaxGap}(s, \sigma)$ is the supremum of distances between consecutive appearances of $\sigma$ in $s$, and $\text{MinGap}(s, \sigma)$ is the infimum of distances between consecutive appearances of $\sigma$ in $s$.

Finally, we define $\text{QR} : \Sigma^\N \times \Sigma \to \N \cup \{ \infty \}$ as follows:
\begin{align*}
  \text{QR}(s, \sigma) = \frac{\text{MaxGap}(s, \sigma)}{\text{MinGap}(s, \sigma)},
\end{align*}
with $\frac{\infty}{\infty} = 1$. We call $\text{QR}(s, \sigma)$ the {\em quasi-regularity} of $\sigma$ in $s$.

In words, for $s\in \Sigma^\N$ and $\sigma \in \Sigma$, $\text{QR}(s, \sigma)$ is the ratio of the largest to smallest distance between consecutive appearances of $\sigma$ in $s$.

Note that, if $\sigma$ appears in $s$ at most once, then $\text{QR}(s, \sigma) = 1$, and if $\sigma$ appears in $s$ at least twice and only finitely many times, or if the distances between consecutive appearances of $\sigma$ can be arbitrarily large, then $\text{QR}(s, \sigma) = \infty$.

\begin{definition}
\label{def:quasi-reg seq}
  Let $s \in \Sigma^\N$. The {\em quasi-regularity} of $s$ is defined as 
  $$\text{QR}(s) = \sup\{\,\text{QR}(s, \sigma) \,|\, \sigma\in\Sigma \}.$$
\end{definition}

In words, the quasi-regularity of a sequence $s\in \Sigma^\N$ is the largest quasi-regularity of $\sigma$ in $s$ for all characters $\sigma\in\Sigma$.

\subsection*{Main Results}
Let $\Sigma$ be a countable alphabet and $\mathbf{p}$ be a probability distribution on $\Sigma$. Define $r_\mathbf{p}$ to be the infimum of all values $\text{QR}(s)$ where $s\in\Sigma^\N$ and $d_s=\mathbf{p}$.

Our first result gives an upper bound for $r_\mathbf{p}$ when $\Sigma$ is finite; $r_\mathbf{p}\leq 2$. On the other hand, when $\mathbf{p}=(1/2,1/3,1/6)$, an easy corollary of ~\cite{kempe2018quasi}*{Proposition 4.6} shows that $r_\mathbf{p} \geq 2$. So, $2$ is the least universal upper bound for $r_\mathbf{p}$ when $\mathbf{p}$ is defined on a finite alphabet.

\begin{theorem}
\label{thm:2-quasi-regular}
  Let $\Sigma$ be a finite alphabet and $\mathbf{p}$ be a probability distribution on $\Sigma$. There exists a sequence $s\in\Sigma^\N$ with $d_s=\mathbf{p}$ and $\text{QR}(s)\leq 2$.
\end{theorem}

Our second result is about the limiting behavior of $r_\mathbf{p}$ when the probability distribution $\mathbf{p}$ tends to 0 uniformly on $\Sigma$; $\lim_{\norm{\mathbf{p}}_\infty\to 0} r_\mathbf{p} = 1$.

\begin{theorem}
\label{thm:eps-quasi-regular}
  Let $\eps>0$. There exists $\delta>0$ such that for any probability distribution $\mathbf{p}$ on a countable alphabet $\Sigma$ with $\norm{\mathbf{p}}_\infty \coloneqq \sup_{\sigma\in\Sigma}\mathbf{p}(\sigma) <\delta$, there exists a sequence $s\in \Sigma^\N$ with $d_s=\mathbf{p}$ and $\text{QR}(s)\leq 1+\eps$.
\end{theorem}

\vspace{0.5cm}
\subsection*{Acknowledgments}
We would like to thank Omer Tamuz for his many useful comments on earlier drafts of this paper.

\section{Main Techniques and Ideas}

\subsubsection*{Theorem ~\ref{thm:2-quasi-regular}}
Let $\Sigma$ be a finite alphabet and $\mathbf{p}$ be a probability distribution on $\Sigma$.

First, note that if $\mathbf{p}$ is binary (i.e.\ if for each $\sigma\in\Sigma$ we have $\mathbf{p}(\sigma) = \frac{1}{2^{k_\sigma}}$ for some $k_\sigma\in\N$), then we can find a reqular periodic sequence $s\in\Sigma^\N$ (i.e.\ a periodic sequence with quasi-regularity equal to 1) with $d_s = \mathbf{p}$. For simplicity, we call the repeated part of $s$ the building block of $s$, which is a finite sequence. So, $s$ is the concatenation of infinitely many copies of its building block.

Next, we observe a statement similar to the following: if $\mathbf{p}$ is almost binary (i.e.\ if for all but at most one $\sigma\in\Sigma$ we have $\mathbf{p}(\sigma) = \frac{1}{2^{k_\sigma}}$ for some $k_\sigma\in\N$), then we can find an almost regular periodic sequence $s\in\Sigma^\N$ (i.e.\ a periodic sequence $s\in\Sigma^\N$ with quasi-regularity equal to either 1 or 2, where for all but at most one $\sigma\in\Sigma$ the quasi-regularity of $\sigma$ in $s$ is equal to 1) with $d_s = \mathbf{p}$. This is the content of Lemma~\ref{lem:n-p-good-strings} part (1).

Almost binary probability distributions play an important role in the proof for two reasons: (I) they are simple enough so that the manipulation of their corresponding almost regular periodic sequences is possible, and (II) any probability distribution, restricted to its support, is a convex combination of almost binary probability distributions.

Now, given a probability distribution $\mathbf{p}$, without loss of generality we may assume $\Sigma$ is its support, and so by (II), we know that $\mathbf{p}$ is a convex combination of some almost binary probability distributions, say $\mathbf{p}_1,\ldots,\mathbf{p}_n$.
For $i=1,\ldots,n$, let $s_i$ be the building block of an almost regular periodic sequence for $\mathbf{p}_i$. By concatenating $m_i\in\N$ copies of $s_i$ and then glueing these chunks by some {\em gluing sequences}, we get a resulting finite sequence where (a) if $m_i$'s are chosen carefully, the density of the resulting sequence is close to $\mathbf{p}$, and (b) in Lemma~\ref{lem:n-p-good-strings} part (2) we show that the gluing sequences can be chosen such that the resulting sequences have quasi-regularity at most 2. So, we can get a finite sequence with quasi-regularity at most 2 and density arbitrarily close to $\mathbf{p}$. Carefully glueing an infinite number of these resulting sequences with densities tending to $\mathbf{p}$, gives us a sequence with quasi-regularity at most 2 and density equal to $\mathbf{p}$.

\subsubsection*{Theorem ~\ref{thm:eps-quasi-regular}}
Let $\Sigma$ be a finite alphabet and $\mathbf{p}$ be a probability distribution on $\Sigma$ with $\norm{\mathbf{p}}_\infty$ close to 0. We generalize the notaion of quasi-regularity so that for $k\in\N$ we can control how regular $k$-distant occurrences of each character are. Lemma~\ref{lemma:composition} allows us to control this notion under compositions, where, roughly speaking, a composition is replacing a certain character with another sequence.

Break $\Sigma$ into some buckets, where in each bucket the ratio of the largest probability to the smallest one is at most 2. For each bucket that has enough characters in it, similar to ~\cite{kempe2018quasi}*{Theorem 6.1}, we use Hall's Marriage theorem in Proposition~\ref{prop:big_1}, and show that there is a sequence with density proportional to the probabilities in the bucket with small quasi-regularity. Let $B_1, B_2,\ldots$ be these buckets, and $s_1,s_2,\ldots$ be the corresponding sequences. Define a new alphabet, where characters are buckets $B_1,B_2,\ldots$, and put a probablity distribution on it, where probablity of $B_i$ is proportional to $\sum_{\sigma\in B_i} \mathbf{p}_\sigma$. For this new alphabet and probability distribution, by Proposition~\ref{prop:big_3}, we have a fairly reqular sequence. Composing this sequence with the $s_i$'s, we get a sequence with density proportional to probabilities of characters in $\cup B_i$ and quasi-regularity close to 1. Let $s_B$ be this sequence.

Let $C$ be the union of the buckets with not enough characters in them, i.e. the ones we did not consider in the previous paragraph. The last step will be adding characters of $C$ to the sequence. For that, we use Hall's Marriage theorem in Proposition~\ref{prop:big_2}, and show that we can open some spaces between characters in $s_B$ and add characters from $C$ in them, so that, the new string has density equal to $\mathbf{p}$ and quasi-regularity close to 1.

\vspace{1cm}
\section{Proof of Theorem ~\ref{thm:2-quasi-regular}}

Since later in some of the definitions and proofs, we need a special character which is not in our alphabet, throughout this proof, we assume the alphabets do not include the character $\sim$. We reserve $\sim$ for this special character.

Here, we introduce some notations and definitions that will help us through the rest of the proof.

\begin{notation}
  For $n,m\in\Z$ let
  \begin{align*}
  	[n] & = \{1, \ldots, n\},\\
    [n, m] & = \{n, n+1, \ldots, m-1, m\},\\
    [n, m) & = \{n, n+1, \ldots, m-1\}.
  \end{align*}
  For $A,B\subseteq \Z$ and $k\in \Z$, let 
  \begin{align*}
    A + B & = \{a+b\mid a\in A, \ b\in B\},\\
    k + A & = \{k + a\mid a\in A\},\\
    k A & = \{k a\mid a\in A\}.
  \end{align*}
\end{notation}

\begin{definition}[Frames on alphabets]
  Let $\Sigma$ be a finite alphabet.
  \begin{enumerate}
  \item A {\em frame} on $\Sigma$ is an $\mathbf{n} \in \N^{\Sigma_\mathbf{n}}$ where $\emptyset \neq \Sigma_\mathbf{n} \subseteq \Sigma$. Define $M_\mathbf{n} \coloneqq \max_{\sigma\in\Sigma_\mathbf{n}}{\mathbf{n}_\sigma+5}$.
  
  \item For a frame $\mathbf{n}\in\N^{\Sigma_\mathbf{n}}$, define $\Pi_\mathbf{n}$ to be the collection of all probability distributions on $\Sigma$ with the following property; for $\sigma\in\Sigma_\mathbf{n}$ we have 
  $$\frac{1}{2^{\mathbf{n}_\sigma}} \leq \mathbf{p}_\sigma \leq \frac{1}{2^{\mathbf{n}_\sigma - 1}}.$$
  
  \item Let $\mathbf{n}\in\N^{\Sigma_\mathbf{n}}$ be a frame on $\Sigma$, and $\mathbf{p} \in \Pi_\mathbf{n}$. We say that $\mathbf{p}$ is a {\em pseudo-binary approximation for $\mathbf{n}$} if for at most one $\sigma\in\Sigma_\mathbf{n}$ we have 
  $\mathbf{p}_\sigma \notin \{\frac{1}{2^{\mathbf{n}_\sigma}}, \frac{1}{2^{\mathbf{n}_\sigma-1}}\}$.
Define $\irr(\mathbf{p}) \coloneqq \sigma$ if such $\sigma\in\Sigma_\mathbf{n}$ exists and $\irr(\mathbf{p}) \coloneqq \ \sim$ otherwise. We denote by $\PBA_\mathbf{n}$ the set of pseudo-binary approximations for the frame $\mathbf{n}$.
  \end{enumerate}
\end{definition}

\begin{claim}
\label{claim:convex-hull}
  If $\mathbf{p}$ is a probability distribution on a finite alphabet $\Sigma$ with $\supp(\mathbf{p}) = \Sigma$, then:
  \begin{enumerate}
    \item There exists $\mathbf{n}\in \N^\Sigma$ with $\mathbf{p}\in\Pi_\mathbf{n}$.
    \item $\Pi_\mathbf{n}$ is equal to the convex hull of $\PBA_\mathbf{n}$ for any frame $\mathbf{n}\in\N^\Sigma$.
    \item $\mathbf{p}$ is in the convex hull of $\PBA_\mathbf{n}$ for some $\mathbf{n}\in \N^\Sigma$.
  \end{enumerate}
\end{claim}

\begin{proof}
  Let $\mathbf{p}$ be a probability distribution on the finite alphabet $\Sigma$ with $\supp(\mathbf{p})=\Sigma$.
  \begin{enumerate}
    \item Note that for each $\sigma \in \Sigma$, there exists an $\mathbf{n}_\sigma\in \N$ with
      $\frac{1}{2^{\mathbf{n}_\sigma}} \leq \mathbf{p}_\sigma \leq \frac{1}{2^{\mathbf{n}_\sigma-1}}$.
      his gives us an $\mathbf{n}\in \N^\Sigma$ with $\mathbf{p}\in\Pi_\mathbf{n}$.
    \item This follows from ~\cite{kempe2018quasi}*{Lemma 4.3}.
    \item This follows from (1) and (2).
  \end{enumerate}
\end{proof}

Now, we generalize some of our previous definitions regarding sequences.

\begin{definition}[Strings]
\label{def:strings}
  Let $\Sigma$ be a finite alphabet.
  \begin{enumerate}
  \item Let $I\subseteq \N$. A {\em string on $I$ with characters from $\Sigma$} is an element of $\Sigma^I$. For $s\in \Sigma^I$, we call $I$ the {\em domain} of $s$ and $\Sigma_s \coloneqq \{s(i)\ | \ i\in I\}$ the set of characters that {\em appear} in $s$.
  
  \item Let $I\subseteq \N$. We define $\text{MaxGap}, \text{MinGap} :\Sigma^I\times\Sigma\to \N\cup\{\infty\}$ as follows:
  \begin{align*}
  \text{MaxGap}(s, \sigma) = \sup\{n \in \N \ | & \ \exists k \in \N \ \text{s.t.} \ k,k+n\in I \\
  & \text{and} \ s_i \neq \sigma \ \text{for} \ i\in (k, k+n)\cap I\}\\
  \text{MinGap}(\alpha, s) = \inf \{n \in \N \ | & \ \exists k \in \N \ \text{s.t.} \ k, k+n \in I \\
  & \text{and} \ s_k = s_{k+n} = \sigma\}
  \end{align*}
  
  \item Let $I\subseteq\N$. We say that $s\in\Sigma^I$ is {\em well-distributed} if for each $\sigma\in\Sigma$ the following limit
  $$\lim_{n \rightarrow \infty}{\frac{\abs{s^{-1}(\sigma) \cap [n]}}{\abs{I \cap [n]}}}$$
  exists. In this case, denote this limit by $d_s(\sigma)$, and call $d_s: \Sigma\to\R$ the density function of $s$. Note that if $s$ is well-distributed, $d_s$ is a probability distribution on $\Sigma$.
  \end{enumerate}
\end{definition}

\begin{definition}[String operations]
Let $\Sigma$ be a finite alphabet and $\mathbf{n}\in \N^{\Sigma_\mathbf{n}}$ be a frame on $\Sigma$.

  \begin{enumerate}
  \item If $I\subseteq \N$ and $s\in\Sigma^I$, for $m\in\Z$ we denote the {\em translate} of $s$ defined on $m+I$ by $\arr{s}{m}$. More precisely, $\arr{s}{m}\in\Sigma^{m+I}$ is defined by $\arr{s}{m}(m+i) = s(i)$ for $i\in I$.
      
  \item If $I,J\subseteq\N$ with $I\cap J=\emptyset$ and $s\in\Sigma^I, t\in\Sigma^J$, we define $s\vee t\in \Sigma^{I\cup J}$ by $(s\vee t)(i)$ equal to $s(i)$ if $i\in I$ and equal to $t(i)$ if $i\in J$.
  
  \item Let $I$ be a finite subset of $\N$. Define $\max_\mathbf{n} I \coloneqq \ceil{\frac{\max I}{2^{M_\mathbf{n}}}}$
  
  \item If $I,J\subseteq\N$ are finite subsets of $\N$ and $s\in\Sigma^I, t\in\Sigma^J$, we define the {\em $\mathbf{n}$-concatenation} of $s$ and $t$ by $s\wedge_\mathbf{n} t \coloneqq s\vee \arr{t}{N}$, where $N=\max_\mathbf{n}I \cdot 2^{M_\mathbf{n}}$.\\
  Define $s^{\wedge_\mathbf{n} k}\coloneqq s\wedge_\mathbf{n} \cdots\wedge_\mathbf{n} s$, where $s$ appears $k$ times.
  
  \item Let $I\subseteq\N$ be finite. For $k\in\N$, define
  $$I[\mathbf{n}, k] \coloneqq \big(-(k-1) 2^{M_\mathbf{n}} + I\big) \cap [2^{M_\mathbf{n}}],$$
  and let $I[\mathbf{n}, -1] \coloneqq I[\mathbf{n}, \max_\mathbf{n}I]$.

  \item Let $I\subseteq\N$ be finite and $s\in\Sigma^I$. For $k\in\N$, let $s[\mathbf{n}, k]\in\Sigma^{I[\mathbf{n}, k]}$ be the restriction of $\arr{s}{-(k-1)2^{M_\mathbf{n}}}$ to $I[\mathbf{n},k]$, and let $s[\mathbf{n}, -1] = s[\mathbf{n}, \max_\mathbf{n}I]$.
  
  \item Let $I,I'\subset\N$ be finite and $s\in\Sigma^I, s'\in\Sigma^{I'}$. If $s[\mathbf{n},-1] = s'[\mathbf{n},+1]$, then we define $s\diamond_\mathbf{n} s' \coloneqq u\wedge_\mathbf{n}s'$, where $u$ is the restriction of $s$ to $[(\max_\mathbf{n}I-1) 2^{M_\mathbf{n}}]$.
  \end{enumerate}
\end{definition}

\begin{definition}[Compatible strings]
\label{def:compatible-strings}
  Let $\Sigma$ be a finite alphabet, $\mathbf{n}\in \N^{\Sigma_\mathbf{n}}$ be a frame on $\Sigma$, and $s \in \Sigma^I$ with $I\subseteq \N$.
  \begin{enumerate}
    \item We say that $s$ is {\em $\mathbf{n}$-compatible}, if for all $\sigma\in\Sigma_\mathbf{n}$ we have:
    $$2^{\mathbf{n}_\sigma-1}\leq \text{MinGap}(s, \sigma)\leq \text{MaxGap}(s, \sigma)\leq 2^{\mathbf{n}_\sigma}.$$
    
    \item We say that $s$ is {\em locally $\mathbf{n}$-compatible}, if for any $\sigma\in \Sigma_\mathbf{n} \cap \Sigma_s$ we have:
    $$2^{\mathbf{n}_\sigma-1}\leq \text{MinGap}(s, \sigma)\leq \text{MaxGap}(s, \sigma)\leq 2^{\mathbf{n}_\sigma}.$$
  \end{enumerate}
\end{definition}

\begin{definition}[String connections]
\label{def:string-graph}
  Let $\Sigma$ be a finite alphabet and $\mathbf{n}\in \N^{\Sigma_\mathbf{n}}$ be a frame on $\Sigma$.
  Let $I, I'\subseteq [2^{M_\mathbf{n}}]$, $J\subseteq \N$ be finite, $s\in\Sigma^I, s'\in\Sigma^{I'}, t\in\Sigma^J$. We write $s\to_{\mathbf{n}, t} s'$ if 
    \begin{itemize}
      \item $\Sigma_t = \Sigma_s = \Sigma_{s'}$,
      \item $t$ is locally $\mathbf{n}$-compatible,
      \item $t[\mathbf{n}, +1] = s$, and
      \item $t[\mathbf{n}, -1] = s'$.
	\end{itemize}
\end{definition}

The proof of the following lemma is straightforward and follows easily from the definitions.

\begin{lemma}
\label{lem:string-graph}
  Let $\Sigma$ be a finite alphabet and $\mathbf{n}\in \N^{\Sigma_\mathbf{n}}$ be a frame on $\Sigma$.
  \begin{enumerate}
    \item Let $I,I',I''\subseteq [2^{M_\mathbf{n}}]$, $J,J'\subseteq \N$ be finite, and $s\in\Sigma^I, s'\in\Sigma^{I'}, s''\in\Sigma^{I''}, t\in\Sigma^J, t'\in\Sigma^{J'}$. If $s\to_{\mathbf{n},t} s'$ and $s'\to_{\mathbf{n},t'}s''$, then $s\to_{\mathbf{n},t\diamond_\mathbf{n} t'}s''$.
    
    \item Let $m\in\N$ and for $1\leq i\leq m$ let $I_i,I_i'\subseteq [2^{M_\mathbf{n}}]$, $J_i\subseteq\N$ be finite, $s_i\in\Sigma^{I_i}, s_i'\in\Sigma^{I_i'}, t_i\in\Sigma^{J_i}$, and $s_i\to_{\mathbf{n},t_i}s_i'$. If
    \begin{itemize}
      \item $\max_\mathbf{n}J_1 = \max_\mathbf{n}J_2 = \cdots = \max_\mathbf{n}J_m$,
      \item $J_1,J_2,\ldots,J_m$ are disjoint, and
      \item $\Sigma_{s_1}\cap\Sigma_\mathbf{n}, \Sigma_{s_2}\cap\Sigma_\mathbf{n}, \ldots, \Sigma_{s_m}\cap\Sigma_\mathbf{n}$ are disjoint,
    \end{itemize}
    then $(s_1\vee\cdots\vee s_m)\to_{\mathbf{n},(t_1\vee\cdots\vee t_m)}(s_1'\vee\cdots\vee s_m')$.
    
    \item Let $J,J'\subseteq\N$ be finite, $t\in\Sigma^J, t'\in\Sigma^{J'}$, and $t[\mathbf{n},-1]=t'[\mathbf{n},+1]$. If $t$ and $t'$ are $\mathbf{n}$-compatible, then $t\diamond_\mathbf{n} t'$ is also $\mathbf{n}$-compatible.
  \end{enumerate}
\end{lemma}

\begin{definition}[Saturated arithmetic progressions]
  Let $\Sigma$ be a finite alphabet.
\begin{itemize}
  \item For $N\in\N$, we say an arithmetic progression $J\subseteq [2^N]$ is {\em $N$-saturated} if the common difference of $J$, which we denote by $\diff_J$, is a power of 2, and $J$ is maximal among all arithmetic progressions with common difference $\diff_J$ in $[2^N]$.
  \item Let $\mathbf{n}\in \N^{\Sigma_\mathbf{n}}$ be a frame on $\Sigma$. An arithmetic progression $J\subseteq [2^{M_\mathbf{n}}]$ is called {\em $\mathbf{n}$-saturated} iff it is $M_\mathbf{n}$-saturated.
\end{itemize}
\end{definition}

\begin{definition}[Uniform strings]
\label{def:uniform-strings}
  Let $\Sigma$ be a finite alphabet, $\mathbf{n}\in \N^{\Sigma_\mathbf{n}}$ be a frame on $\Sigma$, and $\mathbf{p}$ be a probability distribution on $\Sigma$. We say that string $s\in \Sigma^{[2^{M_\mathbf{n}}]}$ is {\em $(\mathbf{n},\mathbf{p})$-uniform} if:
  \begin{itemize}
    \item $d_s=\mathbf{p}$, and
    \item  For any $\sigma\in\Sigma_\mathbf{n}$, there are $\mathbf{n}$-saturated arithmetic progressions $J^1_\sigma$ and $J^2_\sigma$ in $[2^{M_\mathbf{n}}]$ with common differences $2^{\mathbf{n}_\sigma}$ and $2^{\mathbf{n}_\sigma-1}$, such that $J^1_\sigma \subseteq s^{-1}(\sigma)\subseteq J^2_\sigma$.
  \end{itemize}
\end{definition}

\begin{remark}
\label{rmk:n-repeatable}
  Let $\Sigma$ be a finite alphabet, $\mathbf{n}\in \N^{\Sigma_\mathbf{n}}$ be a frame on $\Sigma$, and $p\in \PBA_\mathbf{n}$. Note that if $s$ is an $(\mathbf{n},\mathbf{p})$-uniform string, then $s^{\wedge_\mathbf{n} k}$ is $\mathbf{n}$-compatible for any $k\in\N\cup\{\infty\}$. In particular, $s$ is $\mathbf{n}$-compatible.
\end{remark}

\begin{lemma}
\label{lem:n-p-good-strings}
  Let $\Sigma$ be a finite alphabet, $\mathbf{n}\in \N^\Sigma$ be a frame on $\Sigma$, and $\mathbf{p}, \mathbf{p}' \in \PBA_{\mathbf{n}}$. Then
  \begin{enumerate}
    \item There exists an $(\mathbf{n}, \mathbf{p})$-uniform string $s\in \Sigma^{[2^{M_\mathbf{n}}]}$.
    \item Further, for any $(\mathbf{n}, \mathbf{p})$-uniform string $s\in\Sigma^{[2^{M_\mathbf{n}}]}$, there exists an $(\mathbf{n}, \mathbf{p}')$-uniform string $s'\in \Sigma^{[2^{M_\mathbf{n}}]}$ and $t\in\Sigma^{[k2^{M_\mathbf{n}}]}$ with $k\in\N$ such that $s\to_{\mathbf{n}, t} s'$.\\
    Denote the lexicographically least such $s'$ by $(\mathbf{n}, \mathbf{p}')[s]$, and for such $s'$ the lexicographically least such $t$ by $\mathbf{n}[s\leftrightarrow s']$.
  \end{enumerate} 
\end{lemma}

\begin{remark}
\label{rmk:n-p-good-strings}
  In the context of the previous lemma, since $\mathbf{n}\in\N^\Sigma$ and $\mathbf{p}, \mathbf{p}'\in\PBA_\mathbf{n}$, we get that $\mathbf{p}_\sigma, \mathbf{p}_\sigma' > 0$ for all $\sigma\in\Sigma$. Since $s$ is $(\mathbf{n}, \mathbf{p})$-uniform and $s'$ is $(\mathbf{n}, \mathbf{p}')$-uniform, we should have $\Sigma_s = \Sigma_{s'} = \Sigma$. Moreover, for any $t$ with $s\to_{\mathbf{n}, t} s'$, we should have $\Sigma_t = \Sigma_s = \Sigma$ and $t$ should be locally $\mathbf{n}$-compatible. Since $\Sigma_t=\Sigma$, local $\mathbf{n}$-compatibility of $t$ implies that $t$ is $\mathbf{n}$-compatible.
\end{remark}

Proving this lemma is the main technical effort of the proof of Theorem ~\ref{thm:2-quasi-regular}.
Now, we turn to the proof of Theorem ~\ref{thm:2-quasi-regular}. We will prove Lemma ~\ref{lem:n-p-good-strings} later.

\begin{proof}[Proof of Theorem ~\ref{thm:2-quasi-regular}]
  Let $\Sigma$ be a finite alphabet and $\mathbf{p}$ be a probability distribution on $\Sigma$. Without loss of generality we may assume $\supp(\mathbf{p}) = \Sigma$. We want to find a sequence $s\in \Sigma^\N$ with $d_s=\mathbf{p}$ and $\text{QR}(s)\leq 2$.
  
  By Lemma ~\ref{claim:convex-hull}, $\mathbf{p}$ is in the convex hull of $\PBA_\mathbf{n}$ for some frame $\mathbf{n}\in \N^\Sigma$ on $\Sigma$. So, there are $\mathbf{p}_1,\ldots,\mathbf{p}_m\in\PBA_\mathbf{n}$ and $0\leq\alpha_1,\ldots,\alpha_m\leq 1$ such that $\mathbf{p} = \sum_{i=1}^m \alpha_i \mathbf{p}_i$ and $\sum_{i=1}^m \alpha_i =1$.

  Let $s_m^0$ be an $(\mathbf{n}, \mathbf{p}_m)$-uniform string, which exists by Lemma ~\ref{lem:n-p-good-strings} part (1).
  For $n=1,2,\ldots$ do the following:
  \begin{itemize}
  
    \item Let $s_1^n = (\mathbf{n},\mathbf{p}_1)[s_m^{n-1}]$  and $t_m^{n-1} = \mathbf{n}[s_m^{n-1}\leftrightarrow s_1^n]$, which are defined in Lemma ~\ref{lem:n-p-good-strings} part (2).
    
    \item For $1<i\leq m$, let $s_i^n = (\mathbf{n},\mathbf{p}_i)[s_{i-1}^n]$ and $t_{i-1}^n = \mathbf{n}[s_{i-1}^n\leftrightarrow s_i^n]$, which are defined in Lemma ~\ref{lem:n-p-good-strings} part (2).
    
    \item Let $r_i^n = \floor{\alpha_i n}+1$.
    
    \item Let $t_n = (s_1^n)^{\wedge_\mathbf{n} r_1^n}\diamond_\mathbf{n} t_1^n\diamond_\mathbf{n} (s_2^n)^{\wedge_\mathbf{n} r_2^n}\diamond_\mathbf{n} t_2^n\diamond_\mathbf{n} \cdots \diamond_\mathbf{n} (s_m^n)^{\wedge_\mathbf{n} r_m^n}\diamond_\mathbf{n} t_m^n$.\\
    Note that by Remark~\ref{rmk:n-repeatable} each $(s_i^n)^{\wedge_\mathbf{n} r_i^n}$ is $\mathbf{n}$-compatible, and by Remark ~\ref{rmk:n-p-good-strings} every $t_i^n$ is $\mathbf{n}$-compatible. Hence, by Lemma~\ref{lem:string-graph} part (3), $t_n$ is $\mathbf{n}$-compatible.
    
  \end{itemize}
  Finally, let $t=t_1\diamond_\mathbf{n} t_2\diamond_\mathbf{n} \cdots$. Note that, again, by Lemma~\ref{lem:string-graph} part (3), $t$ is $\mathbf{n}$-compatible, which implies that it is 2-quasi-regular. Since the lengths of $t_i^n$'s are bounded (we have finitely many $s_i^n$'s), we see that $d_t=\sum_{i=1}^m \alpha_i \mathbf{p}_i = \mathbf{p}$.
\end{proof}

\vspace{1cm}
\section{Proof of Lemma ~\ref{lem:n-p-good-strings}}

\subsection{Proof of the First Claim in the Lemma}
Let $\Sigma$ be a finite alphabet. For $N\in\N$, there is a bijection between strings on $[2^N]$ with characters from $\Sigma$, and full-binary rooted trees of depth $N$ with leaves labeled with characters of $\Sigma$. We call such trees {\em $\Sigma$-labeled $N$-deep trees}. Here we explain this bijection. If $T$ is such a tree, for any leaf $v\in T$, there is a unique simple path from the root of the tree, $\troot_T$, to $v$. We can code this path by a sequence of 0's and 1's. Anytime we go to the left child, we put a 0, and anytime we go to the right child, we put a 1. This way we get $d_0,d_1, \ldots, d_{N-1}$, where $d_0$ corresponds to the first step of the path. This determines a string $s\in \Sigma^{[2^N]}$, with $s(\overline{d_{N-1}\cdots d_0}+1)$ equal to the label of $v$ in $T$, where $\overline{d_{N-1}\cdots d_0}= \sum_{i=0}^{N-1} d_i 2^i$.
We denote the tree corresponding to a string $s\in \Sigma^{[2^N]}$ by $\stree(s)$, and the string corresponding to a $\Sigma$-labeled $N$-deep tree $T$ by $\tstring(T)$. This justifies using such strings and such trees interchangeably.

For a $\Sigma$-labeled $N$-deep tree, we can expand the labels on the leaves to more vertices of the tree. Let $T$ be such a tree, $v$ be a vertex in $T$, and $\sigma\in \Sigma$. We give $v$ label $\sigma$ if all the leaves in the subtree of $T$ rooted at $v$ have label $\sigma$.

\begin{remark}
  Note that if $J\subseteq [2^{M_\mathbf{n}}]$ is an $\mathbf{n}$-saturated arithmetic progression then we have a bijection between strings on $J$ with characters from $\Sigma$, and $\Sigma$-labeled $\log_2(\abs{J})$-deep trees.
\end{remark}

In order to work with $\Sigma$-labeled $N$-deep trees easier, we introduce a set of notations and definitions.

\begin{definition}[Trees]
  Let $\Sigma$ be a finite alphabet, $N\in\N$, $T$ be a $\Sigma$-labeled $N$-deep tree, and $v$ be a vertex in $T$.
  \begin{enumerate}
  \item We denote the {\em root} of $T$ by $\troot_T$.
  
  \item If $v\neq \troot_T$, the {\em parent} of $v$ is the first node (after $v$) in the unique simple path from $v$ to $\troot_T$, and the {\em sibling} of $v$ is the unique vertex $w\neq v$ in $T$ which has the same parent as $v$.
  
  \item We denote by $\depth_v$ the {\em depth} of $v$ in $T$, which is equal to the distance between $v$ and $\troot_T$.

  \item We denote by $T_v$ the {\em subtree} of $T$ rooted at $v$. Note that $T_v$ is a $\Sigma$-labeled $(N-\depth_v)$-deep tree.
  
  \item $v$ determines an $N$-saturated arithmetic progression in $[2^N]$. Call this arithmetic progression $I_v$. Obviously $\diff_{I_v} = 2^{\depth_v}$, where $\diff_J$ is the common difference of $J$ when $J$ is an arithmetic progression.
  
  \item If $w$ is also a vertex in $T$ and $\depth_v=\depth_w$, then
  $$d(v\to w) \coloneqq \inf\{d\geq 0\mid d+I_v=I_w \mod 2^N \}.$$
  
  \item We denote by $\tstring(v)$ the restriction of $\tstring(T)$ to $I_v$. So, $\tstring(v)\in\Sigma^{I_v}$ and $\tstring(T) = \tstring(\troot_T)$.

  \item For $\sigma\in\Sigma$ let $\text{num}_T(\sigma)$ be the number of leaves in $T$ with label $\sigma$.
  
  \item We can also define a {\em density function} for such trees; $$d_T(\sigma) = \text{num}_{T}(\sigma)/2^N.$$ We have $d_T=d_{\,\tstring(T)}$.
  
  \item An {\em isomorphism} of $\Sigma$-labeled $N$-deep trees is an isomorphism of binary trees that respects the labels, but does not necessarily respect the order of children for a vertex. When $T^{(1)}$ and $T^{(2)}$ are isomorphic $\Sigma$-labeled $N$-deep trees, we can write $T^{(1)}\cong T^{(2)}$.

  \item Let $\mathbf{n}\in \N^{\Sigma_\mathbf{n}}$ be a frame on $\Sigma$ and assume that $N=M_\mathbf{n}$. We say $v$ is {\em $\mathbf{n}$-isolated in $T$}, if for any character $\sigma\in \Sigma$ that appears both in $T_v$ (i.e.\ at least one of the leaves of $T_v$ has label $\sigma$) and outside of $T_v$ in $T$ (i.e.\ at least one of the leaves of $T$ which is not in $T_v$ has label $\sigma$), we have $\sigma\notin\Sigma_\mathbf{n}$. In other words, no $\sigma\in\Sigma_\mathbf{n}$ appears both in $T_v$ and outside of $T_v$.
  \end{enumerate}
\end{definition}

Here, parallel to Definitions ~\ref{def:compatible-strings} and ~\ref{def:uniform-strings}, which are for strings, we define compatible and uniform trees.

\begin{definition}[Compatible trees]
  Let $\Sigma$ be a finite alphabet, $\mathbf{n}\in \N^{\Sigma_\mathbf{n}}$ be a frame on $\Sigma$, and $T$ be a $\Sigma$-labeled $M_\mathbf{n}$-deep tree. We say that $T$ is {\em $\mathbf{n}$-compatible} if $\tstring(T)$ is $\mathbf{n}$-compatible.
\end{definition}

\begin{definition}[Uniform trees]
\label{def:n-p-uniform-tree}
  Let $\Sigma$ be a finite alphabet, $\mathbf{n}\in \N^{\Sigma_\mathbf{n}}$ be a frame on $\Sigma$, $\mathbf{p}$ be a probability distribution on $\Sigma$, and $T$ be a $\Sigma$-labeled $M_\mathbf{n}$-deep tree. We say that $T$ is {\em $(\mathbf{n},\mathbf{p})$-uniform} if
  \begin{itemize}
    \item $d_T=\mathbf{p}$, and
    
    \item For any $\sigma\in\Sigma_\mathbf{n}$, there is a shallowest vertex $v$ (i.e.\ vertex with the least depth) in $T$ labeled $\sigma$, such that any leaf $x\in T$ labeled $\sigma$ is either in $T_v$ or in $T_w$, where $w$ is $v$'s sibling.\\
    In this case, note that such $v$ is unique and set $c_T(\sigma)=v$
  \end{itemize}
\end{definition}

Let $\Sigma$ be a finite alphabet, $\mathbf{n}\in \N^{\Sigma_\mathbf{n}}$ be a frame on $\Sigma$, and $\mathbf{p}\in\PBA_\mathbf{n}$. Note that $(\mathbf{n},\mathbf{p})$-uniform trees are also $\mathbf{n}$-compatible. Moreover, it follows from the definitions that $s\in\Sigma^{[2^{M_\mathbf{n}}]}$ is $(\mathbf{n}, \mathbf{p})$-uniform iff $\stree(s)$ is $(\mathbf{n}, \mathbf{p})$-uniform.

\begin{proof}[Proof of Lemma~\ref{lem:n-p-good-strings} part (1)]
  Let $\Sigma$ be a finite alphabet, $\mathbf{n}\in \N^\Sigma$ be a frame on $\Sigma$, and $\mathbf{p}\in\PBA_\mathbf{n}$. By the previous paragraph, to prove Lemma~\ref{lem:n-p-good-strings} part (1), we need to show that there exists an $(\mathbf{n}, \mathbf{p})$-uniform tree. For that, follow this procedure:
  \begin{enumerate}
  \item If $\sigma = \irr(\mathbf{p})\neq \ \sim$, write $\mathbf{p}_\sigma$ as a sum of different powers of two. For each power of two in the sum, say $2^z$, add a character with density equal to $2^z$ to $\Sigma$. Call these starred characters. Finally remove $\sigma$ from $\Sigma$.
  
  \item Make the Huffman coding of $\Sigma$ for the given densities with one modification: when you choose the two smallest characters in each step, always give priority to the starred characters. If one of the two smallest characters is starred, consider the new character starred, too.
  
  \item This way we obtain a binary tree. In this tree change the labels of all the starred leaves to $\sigma$. Expand this tree to get a full-binary rooted tree of depth $M_\mathbf{n}$. Call this tree $T$. It is not difficult to see that $T$ is $(\mathbf{n}, \mathbf{p})$-uniform.
  \end{enumerate}
\end{proof}

To show the second part of Lemma~\ref{lem:n-p-good-strings}, we need to build more tools that allow us manipulate $\Sigma$-labeled $M_\mathbf{n}$-deep trees.

\subsection{Expandable, Contractible, and Super-Contractible Vertices}

\begin{definition}
  Let $\Sigma$ be a finite alphabet, $\mathbf{n}\in \N^{\Sigma_\mathbf{n}}$ be a frame on $\Sigma$, and $J\subseteq [2^{M_\mathbf{n}}]$ be an $\mathbf{n}$-saturated arithmetic progression. Note that we have
  $\abs{J}=2^{M_\mathbf{n}}/\diff_J$. Let $a = \min J$.
  
  For $i\in [0,\abs{J}]$ and $k, d\in \N$ define:
  \begin{align*}
    \DOM^\infty_{\mathbf{n}}(J, i, k, d) =& \ \Big( a + \diff_J\cdot [0,k|J|+i) \Big) \\
    \cup &\ \Big(a+ \diff_J (k\abs{J}+i-1) + d + \diff_J\cdot[0,+\infty) \Big) \\
    \DOM_{\mathbf{n}}(J, i, k, d) =& \ \DOM^\infty_{\mathbf{n}}(J,i,k,d) \cap \big[(k+3)2^{M_\mathbf{n}}\big]\\
    \STR_\mathbf{n}(J, i, k, d) =& \ \DOM_{\mathbf{n}}(J, i, k, d)[\mathbf{n}, +1]\\
    \END_\mathbf{n}(J, i, k, d) =& \ \DOM_{\mathbf{n}}(J, i, k, d)[\mathbf{n}, -1]
  \end{align*}
\end{definition}

Here we explain $A = \DOM^\infty_\mathbf{n}(J,i,k,d)$ in words. $A$ is almost an arithmetic progression; it starts with the first element of $J$ and continues as an arithmetic progression with the common difference $\diff_J$ for $k\abs{J}+i$ elements. Let $x$ be the last element in our arithmetic progression so far. The next element of $A$ is equal to $x+d$, and after that, again, $A$ continues as an arithmetic progression with the common difference $\diff_J$. In other words, $A$ consists of two arithmetic progressions with the common difference $\diff_J$ which are concatenated by the difference $d$.

We can see that $\STR_\mathbf{n}(J,i,k,d)$ is the initial segment of $\DOM_\mathbf{n}(J,i,k,d)$, which is equal to $J$; and $\END_\mathbf{n}(J,i,k,d)$ is the terminal segment of $\DOM_\mathbf{n}(J,i,k,d)$, which is equal to $d+J$ mod $2^{M_\mathbf{n}}$.

\begin{definition}[Relaxable strings]
\label{def:relaxable}
  Let $\Sigma$ be a finite alphabet, $\mathbf{n}\in \N^{\Sigma_\mathbf{n}}$ be a frame on $\Sigma$, $J\subseteq [2^{M_\mathbf{n}}]$ be an $\mathbf{n}$-saturated arithmetic progression, and $s\in \Sigma^J$.
  For $i\in [0,\abs{J}]$ and $d\in \N$, we say that $s$ is {\em $(i,d)_\mathbf{n}$-relaxable} if there are $k\in \N$, $u\in\Sigma^{\END}$, and $t\in \Sigma^{\DOM}$ such that $s\to_{\mathbf{n}, t} u$ and $\stree(u)\cong \stree(s)$ as $\Sigma$-labeled $\log_2(\abs{J})$-deep trees, where $\DOM=\DOM_{\mathbf{n}}(J, i, k, d)$ and $\END=\END_{\mathbf{n}}(J, i, k, d)$.
\end{definition}

\begin{definition}[Expandable, contractible, and super-contractible vertices]
  Let $\Sigma$ be a finite alphabet, $\mathbf{n}\in \N^{\Sigma_\mathbf{n}}$ be a frame on $\Sigma$, and $T$ be a $\Sigma$-labeled $M_\mathbf{n}$-deep tree. Let $v$ be a vertex in $T$ and let $W\subseteq T$ be a subset of vertices in $T$. Set $I_W = \cup_{w\in W}I_w$.
  \begin{enumerate}  
    \item We say that $v$ is {\em $\mathbf{n}$-expandable with respect to $W$}, if $I_v\subseteq I_W$ and $\tstring(v)$ is $(i,d)_\mathbf{n}$-relaxable for all $i\in [0,\abs{I_v}]$ and $\diff_{I_v} \leq d \leq 2 \ \diff_{I_v}$ with $d+I_v \subseteq I_W \mod 2^{M_{\mathbf{n}}}$.
    
    \item We say that $v$ is {\em $\mathbf{n}$-contractible with respect to $W$}, if $I_v\subseteq I_W$ and $\tstring(v)$ is $(i,d)_\mathbf{n}$-relaxable for all $i\in [0,\abs{I_v}]$ and $\diff_{I_v}/2 \leq d \leq \diff_{I_v}$ with $d+I_v \subseteq I_W \mod 2^{M_{\mathbf{n}}}$.
    
    \item We say that $v$ is {\em $\mathbf{n}$-super-contractible with respect to $W$}, if $I_v\subseteq I_W$ and $\tstring(v)$ is $(i,d)_\mathbf{n}$-relaxable for all $i\in [0,\abs{I_v}]$ and $1 \leq d \leq \diff_{I_v}$ with $d+I_v \subseteq I_W \mod 2^{M_{\mathbf{n}}}$.
  \end{enumerate}
\end{definition}

Note that $\mathbf{n}$-super-contractible vertices with respect to a given subset of $T$ are also $\mathbf{n}$-contractible with respect to that subset. Also, if $v$ is $\mathbf{n}$-(expandable/contractible/super-contractible) with respect to $W$, and $I_v\subseteq I_{W'}\subseteq I_W$, then $v$ is also $\mathbf{n}$-(expandable/contractible/super-contractible) with respect to $W'$.

\begin{proposition}
\label{prop:expandable-contractible-leaf}
  Let $\Sigma$ be a finite alphabet, $\mathbf{n}\in \N^{\Sigma_\mathbf{n}}$ be a frame on $\Sigma$, and $T$ be a $\Sigma$-labeled $M_\mathbf{n}$-deep tree. Assume that $v\in T$ is a vertex with label $\sigma$. The following hold.
  \begin{enumerate}
    \item If $\sigma\in\Sigma_\mathbf{n}$ and $\diff_{I_v} = 2^{\mathbf{n}_\sigma-1}$, then $v$ is $\mathbf{n}$-expandable with respect to $\{\troot_T\}$.
    
    \item If $\sigma\in\Sigma_\mathbf{n}$ and $\diff_{I_v} = 2^{\mathbf{n}_\sigma}$, then $v$ is $\mathbf{n}$-contractible with respect to $\{\troot_T\}$.
    
    \item If $\sigma\notin\Sigma_\mathbf{n}$, then $v$ is both $\mathbf{n}$-expandable and $\mathbf{n}$-super-contractible with respect to $\{\troot_T\}$.
  \end{enumerate}
\end{proposition}
\begin{proof}
    We will show (1). The proofs of (2) and (3) are very similar. Let $s = \tstring(v)$. We need to show that $s$ is $(i,d)_\mathbf{n}$-relaxable for $i\in[0,\abs{I_v}]$ and $d\in\N$ with $\diff_{I_v}\leq d\leq 2\,\diff_{I_v}$. We show that $k=2$ works in Definition ~\ref{def:relaxable}. Let $\DOM=\DOM_{\mathbf{n}}(I_v, i, 2, d)$, $\STR=\STR_{\mathbf{n}}(I_v, i, 2, d)$, $\END=\END_{\mathbf{n}}(I_v, i, 2, d)$.\\
    Let $t\in \Sigma^{\DOM}, u\in\Sigma^{\END}$ be strings with all characters equal to $\sigma$. Since $\stree(u)$ and $\stree(s) = T_v$ have the same depth, and all vertices in both of them are labeled $\sigma$, we have $\stree(u) \cong \stree(s)$. So, all we need to show is that $s\to_{\mathbf{n}, t} u$.
    \begin{itemize}
      \item $\Sigma_s = \Sigma_u = \Sigma_t = \{\sigma\}$.
      
      \item The gaps between appearances of $\sigma$ in $t$ are either $\diff_{I_v}$ or $d$. We know that $\diff_{I_v} = 2^{\mathbf{n}_\sigma-1},$ and
      $$2^{\mathbf{n}_\sigma-1} = \diff_{I_v} \leq d \leq 2\,\diff_{I_v} = 2^{\mathbf{n}_\sigma}.$$
      So, the gaps are bounded between $2^{\mathbf{n}_\sigma-1}$ and $2^{\mathbf{n}_\sigma}$. Since $\Sigma_t = \{\sigma\}$, this shows that $t$ is locally $\mathbf{n}$-compatible.
      
      \item Obviously, $t[\mathbf{n}, +1]$ is the constant $\sigma$ string on $\STR = I_v$, which is equal to $s$.
      \item Similarly, $t[\mathbf{n}, -1]$ is the constant $\sigma$ strong on $\END$, which is equal to $u$.
    \end{itemize}
So, this shows that $s\to_{\mathbf{n}, t} u$. Hence $s$ is $(i,d)_\mathbf{n}$-relaxable for all $i\in[0,\abs{I_v}]$ and $\diff_{I_v} \leq d \leq 2\,\diff_{I_v}$, which means $v$ is $\mathbf{n}$-expandable with respect to $\{\troot_T\}$.
\end{proof}

\begin{proposition}
\label{prop:expandable-contractible-parent}
  Let $\Sigma$ be a finite alphabet, $\mathbf{n}\in \N^{\Sigma_\mathbf{n}}$ be a frame on $\Sigma$, and $T$ be a $\Sigma$-labeled $M_\mathbf{n}$-deep tree. Assume that $x,y,z\in T$, $W\subseteq T$, where $x, y$ are the two different children of $z$ and each of $x$ and $y$ is either $\mathbf{n}$-expandable or $\mathbf{n}$-contractible with respect to $W$. Moreover, assume that $\Sigma_{\tstring(x)}\cap\Sigma_\mathbf{n}$ and $\Sigma_{\tstring(y)}\cap\Sigma_\mathbf{n}$ are disjoint. 
  \begin{enumerate}
    \item If $x$ and $y$ are both $\mathbf{n}$-contractible with respect to $W$, then $z$ is $\mathbf{n}$-super-contractible with respect to $W$.
    
    \item If either of $x$ or $y$ is $\mathbf{n}$-expandable with respect to $W$, then $z$ is also $\mathbf{n}$-expandable with respect to $W$.
    
    \item If either of $x$ or $y$ is $\mathbf{n}$-super-contractible with respect to $W$, then $z$ is also $\mathbf{n}$-super-contractible with respect to $W$.
  \end{enumerate}
\end{proposition}
\begin{proof}
  For $i\in[0,\abs{I_z}]$, we say that $i$ {\em belongs} to $x$ if $(\min I_z + i\,\diff_{I_z})\in I_x \mod 2^{M_\mathbf{n}}$, and similarly we say that $i$ belongs to $y$ if $(\min I_z + i\,\diff_{I_z})\in I_y \mod 2^{M_\mathbf{n}}$. Note that each $i\in[0,\abs{I_z}]$ belongs to exactly one of $x$ or $y$.
  
  Let $s_x = \tstring(x), s_y = \tstring(y), s_z = \tstring(z)$.

  \vspace{0.5cm}
  {\em Proof of (1).}
  To show that $z$ is $\mathbf{n}$-super-contractible with respect to $W$, we need to show that $I_z\subseteq I_W$ and $s_z$ is $(i_z,d_z)_\mathbf{n}$-relaxable for any $i_z\in[0,\abs{I_z}]$ and $1\leq d_z\leq \diff_{I_z}$ with $d_z+I_z\subseteq I_W \mod 2^{M_{\mathbf{n}}}$. Note that since $I_x, I_y\subseteq I_W$, we have $I_z=I_x\cup I_y \subseteq I_W$. So, let $i_z\in[0,\abs{I_z}]$ and $1\leq d_z\leq \diff_{I_z}$ with $d_z+I_z\subseteq I_W \mod 2^{M_{\mathbf{n}}}$. Without loss of generality we may assume that $i_z$ belongs to $x$.
  
  Let $d_x = d_y = d_z+\diff_{I_z}$, $i_x = \floor{\frac{i_z}{2}}$, and $i_y=\floor{\frac{i_z+1}{2}}$.
  \begin{itemize}
    \item $i_x\in [0, \abs{I_x}]$.
    \item Since $1\leq d_z\leq \diff_{I_z}$, we get that $\diff_{I_x}/2\leq d_x\leq\diff_{I_x}$.
    \item $d_x + I_x = (d_z + \diff_{I_z}) + I_x \subseteq (d_z+\diff_{I_z}) + I_z = d_z + I_z \subseteq I_W \mod 2^{M_\mathbf{n}}$.
  \end{itemize}
  So, because $x$ is $\mathbf{n}$-contractible with respect to $W$, there are $k_x\in \N, u_x\in\Sigma^{\END_x}, t_x\in \Sigma^{\DOM_x}$ that satisfy the statement in Definition ~\ref{def:relaxable}, where $\DOM_x = \DOM_\mathbf{n}(I_x, i_x, k_x, d_x)$ and $\STR_x,\END_x$ are defined similarly. A similar argument shows there are $k_y\in \N, u_y\in\Sigma^{\END_y}, t_y\in \Sigma^{\DOM_y}$ that satisfy the statement in Definition ~\ref{def:relaxable}, where $\DOM_y = \DOM_\mathbf{n}(I_y, i_y, k_y, d_y)$ and $\STR_y,\END_y$ are defined similarly.
  
  Wihout loss of generality, we may assume that $k_x=k_y$; this is because $(s_x\wedge_\mathbf{n}t_x) \in\Sigma^{\DOM_\mathbf{n}(I_x, i_x, k_x+1, d_x)}$ and $s_x\to_{\mathbf{n},(s_x\wedge_\mathbf{n}t_x)} u_x$, and a similar result holds for $y$. So, let $k_z \coloneqq k_x = k_y$. Let $\DOM_z = \DOM_\mathbf{n}(I_z, i_z, k_z, d_z)$ and define $\STR_z,\END_z$ similarly.
  
  \begin{itemize}
    \item Since $k_x=k_y$, we have $\max_\mathbf{n}\DOM_x = \max_\mathbf{n}\DOM_y$.
    \item $\DOM_x\cap\,\DOM_y =\emptyset$ and $\DOM_x\cup\,\DOM_y= \DOM_z$ (similar results hold for $\STR$ and $\END$).
    \item From the assumptions in the proposition, we know that $\Sigma_{s_x}\cap\Sigma_\mathbf{n}$ and $\Sigma_{s_y}\cap\Sigma_\mathbf{n}$ are disjoint.
  \end{itemize}
   So, by Lemma~\ref{lem:string-graph} part (2), we get $(s_x\vee s_y)\to_{\mathbf{n},(t_x\vee t_y)}(u_x\vee u_y)$. Note that $s_z=s_x\vee s_y$. Let $t_z\coloneqq t_x\vee t_y$ and $u_z\coloneqq u_x\vee u_y$. Since $\stree(u_x)\cong T_x$ and $\stree(u_y)\cong T_y$, we obviously have $\stree(u_z)\cong T_z$. This shows that $s_z$ is $(i_z,d_z)_\mathbf{n}$-relaxable, which completes the proof.
  
  \vspace{0.5cm}
  {\em Proof of (2).}
  If $x$ and $y$ are both $\mathbf{n}$-expandable with respect to $W$, the proof is similar to part (1), and we do not repeat it here. So, without loss of generality assume that $x$ is $\mathbf{n}$-expandable and $y$ is $\mathbf{n}$-contractible with respect to $W$. To show that $z$ is $\mathbf{n}$-expandable with respect to $W$, we need to show that $I_z\subseteq I_W$ and $s_z$ is $(i_z,d_z)_\mathbf{n}$-relaxable for any $i_z\in[0, \abs{I_z}]$ and $\diff_{I_z}\leq d_z\leq 2\,\diff_{I_z}$ with $d_z+I_z\subseteq I_W \mod 2^{M_{\mathbf{n}}}$. Note that since $I_x, I_y\subseteq I_W$, we have $I_z=I_x\cup I_y \subseteq I_W$. So, let $i_z\in[0, \abs{I_z}]$ and $\diff_{I_z}\leq d_z\leq 2\,\diff_{I_z}$ with $d_z+I_z\subseteq I_W \mod 2^{M_{\mathbf{n}}}$. Now, consider the following two cases:
  
  \vspace{0.2cm}
  {\em Case (I): $i_z$ belongs to $x$.}\\
  Let $d_x=d_z+2\,\diff_{I_z}$, $d_y=d_z$, $i_x=\floor{\frac{i_z}{2}}$, and $i_y = \floor{\frac{i_z+1}{2}}$.
  \begin{itemize}
    \item $i_x\in [0, \abs{I_x}]$.
    \item Since $\diff_{I_z}\leq d_z\leq 2\,\diff_{I_z}$, we get that $\diff_{I_x} \leq d_x \leq 2\,\diff_{I_x}$.
    \item $d_x + I_x = (d_z + 2\,\diff_{I_z}) + I_x \subseteq (d_z+2\,\diff_{I_z}) + I_z = d_z + I_z \subseteq I_W \mod 2^{M_\mathbf{n}}$.
  \end{itemize}
  Since $x$ is $\mathbf{n}$-expandable with respect to $W$, there are $k_x\in \N, u_x\in\Sigma^{\END_x}, t_x\in\Sigma^{\DOM_x}$ that satisfy the statement in Definition ~\ref{def:relaxable}, where $\DOM_x = \DOM_\mathbf{n}(I_x,i_x,k_x,d_x)$ and $\STR_x,\END_x$ are defined similarly. Similarly, we can see there are $k_y\in \N, u_y\in\Sigma^{\END_y}, t_y\in\Sigma^{\DOM_y}$ that satisfy the statement in Definition ~\ref{def:relaxable}, where $\DOM_y = \DOM_\mathbf{n}(I_y,i_y,k_y,d_y)$ and $\STR_y,\END_y$ are defined similarly. Like the proof of part (I), without loss of generality we may assume $k_x = k_y$.
  
  Let $k_z\coloneqq k_x = k_y$, $\DOM_z = \DOM_\mathbf{n}(I_z, i_z, k_z, d_z)$, and $\STR_z,\END_z$ be defined similarly.
  
  \begin{itemize}
    \item Since $k_x = k_y$, we have $\max_\mathbf{n}\DOM_x = \max_\mathbf{n}\DOM_y$.
    \item $\DOM_x\cap\,\DOM_y =\emptyset$ and $\DOM_x\cup\,\DOM_y= \DOM_z$ (similar results hold for $\STR$ and $\END$).
    \item From the assumptions in the proposition, we know that $\Sigma_{s_x}\cap\Sigma_\mathbf{n}$ and $\Sigma_{s_y}\cap\Sigma_\mathbf{n}$ are disjoint.
  \end{itemize}
  So, by Lemma~\ref{lem:string-graph} part (2), we get $(s_x\vee s_y)\to_{\mathbf{n},(t_x\vee t_y)}(u_x\vee u_y)$. Note that $s_z=s_x\vee s_y$. Let $t_z\coloneqq t_x\vee t_y$ and $u_z\coloneqq u_x\vee u_y$. Since $\stree(u_x)\cong T_x$ and $\stree(u_y)\cong T_y$, we obviously have $\stree(u_z)\cong T_z$. This shows that $s_z$ is $(i_z,d_z)_\mathbf{n}$-relaxable, which completes the proof in case (I).
  
\begin{remark}
\label{rmk:left-right}
  Let $\stree(u_z)_{\text{left}}$ and $\stree(u_z)_{\text{right}}$ be the subtrees rooted at the left and the right children of $\troot_{\stree(u_z)}$. Let $(T_z)_{\text{left}}, (T_z)_{\text{right}}$ be defined similarly. We showed in the proof that $T_z\cong \stree(u_z)$, so, either
    $$(T_z)_{\text{left}}\cong \stree(u_z)_{\text{left}}, (T_z)_{\text{right}}\cong \stree(u_z)_{\text{right}},$$
    or
    $$(T_z)_{\text{left}}\cong \stree(u_z)_{\text{right}}, (T_z)_{\text{right}}\cong \stree(u_z)_{\text{left}}.$$
    If $d_z=\diff_{I_z}$, our construction guarantees that $$(T_z)_{\text{left}}\cong \stree(u_z)_{\text{right}}, (T_z)_{\text{right}}\cong \stree(u_z)_{\text{left}}.$$
  \end{remark}
  
  \vspace{0.2cm}
  {\em Case (II): $i_z$ belongs to $y$.}\\
  Pick some $i^{(0)}_z\in [0, \abs{I_z}]$ that belongs to $x$ (either $i^{(0)}_z=0$ or $i^{(0)}_z=1$ works), and let $d^{(0)}_z=\diff_{I_z}$. The proof in case (I) applied to $(i^{(0)}_z,d^{(0)}_z)_\mathbf{n}$ for $s_z$ gives us $k^{(0)}_z\in\N$, $u^{(0)}_z\in\Sigma^{\END^{(0)}_z}, t^{(0)}_z\in\Sigma^{\DOM^{(0)}_z}$, where $\DOM^{(0)}_z=\DOM_\mathbf{n}(I_z,i^{(0)}_z,k^{(0)}_z,d^{(0)}_z)$ and $\STR^{(0)}_z, \END^{(0)}_z$ are defined similarly, such that 
  \begin{equation}
  \label{eq:ec_parent_phase_0}
    \stree(u_z^{(0)})\cong T_z \text{ and } s_z\to_{\mathbf{n}, t_z^{(0)}} u_z^{(0)}.
  \end{equation}
  Moreover, using the notation from Remark ~\ref{rmk:left-right}, we have that
      $$\stree(u_z^{(0)})_{\text{left}}\cong (T_z)_{\text{right}}, \stree(u_z^{(0)})_{\text{right}}\cong (T_z)_{\text{left}}.$$
      So, for example, if $x$ is the left child of $z$, now $\stree(u_z^{(0)})_{\text{right}}$ is isomorphic to $T_x$.
  
  Let $s^{(1)}\in\Sigma^{[2^{M_\mathbf{n}}]}$ be almost equal to $\tstring(T)$ with a change: the restriction of $s^{(1)}$ to $I_z$ is equal to $u_z^{(0)}$. Let $T^{(1)} = \stree(s^{(1)})$. In other words, $T^{(1)}$ is equal to $T$ when $T_z$ is replaced by $\stree(u_z^{(0)})$. For any $v\in T$, let $v^{(1)}\in T^{(1)}$ be the corresponding vertex in $T^{(1)}$ to $v$, i.e.\ $I_{v^{(1)}} = I_v$, and let $W^{(1)}$ be the subset of vertices in $T^{(1)}$ corresponding to $W$. Note that the subtree of $T^{(1)}$ rooted at $x^{(1)}$, i.e.\ $T^{(1)}_{x^{(1)}}$, is isomorphic to $T_y$, and, similarly, $T^{(1)}_{y^{(1)}}$ is isomorphic to $T_x$. Moreoever, there is an isomorphism $\phi:T\to T^{(1)}$ such that $\phi(y)=x^{(1)}$, $\phi(x)=y^{(1)}$, and $\phi(v)=v^{(1)}$ for every $v\in T\setminus(T_x\cup T_y)$. Note that $I_{\phi(W)} = I_W = I_{W^{(1)}}$.
  
  Since $x$ is $\mathbf{n}$-expandable with respect to $W$, we get that $y^{(1)} = \phi(x)$ is $\mathbf{n}$-expandable with respect to $\phi(W)$, and since $I_{\phi(W)} = I_{W^{(1)}}$, also with respect to $W^{(1)}$. Similarly, $x^{(1)}$ is $\mathbf{n}$-contractible with respect to $W^{(1)}$. Since we assumed $i_z$ belongs to $y$, we also have that $i_z$ belongs to $y^{(1)}$. Now, we can apply case (I) to $x^{(1)}, y^{(1)}, z^{(1)}, W^{(1)}, T^{(1)}$ and $(i^{(1)}_z=i_z, d^{(1)}_z=d_z)_\mathbf{n}$ to get $k^{(1)}_z\in\N, u_z^{(1)}\in\Sigma^{\END^{(1)}_z}, t_z^{(1)}\in\Sigma^{\DOM^{(1)}_z}$, where $\DOM^{(1)}_z=\DOM_\mathbf{n}(I_{z^{(1)}},i^{(1)}_z,k^{(1)}_z,d^{(1)}_z)$ and $\STR^{(1)}_z,\END^{(1)}_z$ are defined similarly, such that 
  \begin{equation}
  \label{eq:ec_parent_phase_1}
    \stree(u_z^{(1)}) \cong \stree(u_z^{(0)}) \text{ and } u_z^{(0)}\to_{\mathbf{n},t_z^{(1)}} u_z^{(1)}.
  \end{equation}
  
  \begin{itemize}
    \item By ~\eqref{eq:ec_parent_phase_0} and ~\eqref{eq:ec_parent_phase_1} we have $\stree(u_z^{(1)}) \cong \stree(u_z^{(0)}) \cong T_z$.
    \item Let $k_z = k^{(0)}_z + 2 + k^{(1)}_z$, $\DOM_z = \DOM_{\mathbf{n}}(I_z,i_z,k_z,d_z)$, and $t_z = t_z^{(0)} \diamond_\mathbf{n} t_z^{(1)}$. Note that $t_z\in\Sigma^{\DOM_z}$. So, by ~\eqref{eq:ec_parent_phase_0}, ~\eqref{eq:ec_parent_phase_1}, and Lemma~\ref{lem:string-graph} part (1), we get that $s_z\to_{\mathbf{n},t_z}u^{(1)}_z$.
  \end{itemize}
  This shows that $s_z$ is $(i_z,d_z)_\mathbf{n}$-relaxable, which completes the proof of case (II). Hence part (2) is proved.

  \vspace{0.5cm}
  {\em Proof of (3).}
  If $x$ and $y$ are both $\mathbf{n}$-contractible with respect to $W$, we know the result from part (1). So, without loss of generality assume that $x$ is $\mathbf{n}$-super-contractible with respect to $W$ and $y$ is $\mathbf{n}$-expandable with respect to $W$. The proof in this case is very similar to that of part (2) and we do not repeat it here.
  
\end{proof}

\begin{proposition}
\label{prop:expandable-contractible-isolated}
    Let $\Sigma$ be a finite alphabet, $\mathbf{n}\in \N^{\Sigma_\mathbf{n}}$ be a frame on $\Sigma$, and $T$ be an $(\mathbf{n},d_T)$-uniform tree. If $z\in T$ is $\mathbf{n}$-isolated in $T$, we have the following.
  \begin{enumerate}
    \item $z$ is either $\mathbf{n}$-expandable or $\mathbf{n}$-contractible with respect to $\{\troot_T\}$.
    \item If a $\sigma\in\Sigma$ with either (i) $\sigma\in\Sigma_\mathbf{n}$ and $d_T(\sigma)>2^{-\mathbf{n}_\sigma}$, or (ii) $\sigma\notin\Sigma_\mathbf{n}$, appears in $T_z$, i.e.\ $\sigma\in\Sigma_{\tstring(z)}$, then $z$ is $\mathbf{n}$-expandable with respect to $\{\troot_T\}$.
    \item If a $\sigma\in\Sigma\setminus\Sigma_\mathbf{n}$ appears in $T_z$, and for any $\beta\in\Sigma_\mathbf{n}$ that appears in $T_z$ we have 
    $d_T(\beta)\in\{2^{-\mathbf{n}_\beta},2^{-(\mathbf{n}_\beta-1)}\}$,
    then $z$ is both $\mathbf{n}$-expandable and $\mathbf{n}$-super-contractible with respect to $\{\troot_T\}$.\\
    In particular, if $d_T\in \PBA_\mathbf{n}$ and $\irr(d_T)$ does not appear in $T_z$ and a $\sigma\in\Sigma\setminus\Sigma_\mathbf{n}$ appears in $T_z$, then $z$ is both $\mathbf{n}$-expandable and $\mathbf{n}$-super-contractible with respect to $\{\troot_T\}$.
  \end{enumerate}
\end{proposition}
\begin{proof}
  We prove this proposition by backward induction on the depth of $z$. If $z$ is a leaf, the proof is easy and similar to the proof of Proposition~\ref{prop:expandable-contractible-leaf}.\\
  Let $x$ and $y$ be $z$'s children. We know that the proposition holds for $x$ and $y$. Since $z$ is $\mathbf{n}$-isolated in $T$ and $T$ is an $(\mathbf{n},d_T)$-uniform tree, we know that either both $x$ and $y$ are $\mathbf{n}$-isolated in $T$ or one of $x$ and $y$ has a label. Let $s_z = \tstring(z), s_x = \tstring(x), s_y = \tstring(y)$.
  
  \vspace{0.5cm}
  {\em Case (I): $x$ and $y$ are both $\mathbf{n}$-isolated in $T$.}\\
  In this case the proposition follows easily from the induction hypothesis for $x$ and $y$ and Proposition~\ref{prop:expandable-contractible-parent}. Note that since $x$ and $y$ are both $\mathbf{n}$-isolated in $T$, we know that $\Sigma_{s_x}\cap\Sigma_\mathbf{n}$ and $\Sigma_{s_y}\cap\Sigma_\mathbf{n}$ are disjoint.
  \begin{enumerate}
    \item If either of $x$ or $y$ is $\mathbf{n}$-expandable with respect to $\{\troot_T\}$, then by Proposition~\ref{prop:expandable-contractible-parent}, so is $z$. If, on the other hand $x$ and $y$ are both not $\mathbf{n}$-expandable, then by the first part of the induction hypothesis, $x$ and $y$ are both $\mathbf{n}$-contractible, and so by Proposition~\ref{prop:expandable-contractible-parent}, $z$ is $\mathbf{n}$-contractible with respect to $\{\troot_T\}$.
    
    \item If a $\sigma\in\Sigma$ appears in $T_z$ with either (i) $\sigma\in\Sigma_\mathbf{n}$ and $d_T(\sigma)>2^{-\mathbf{n}_\sigma}$, or (ii) $\sigma\notin\Sigma_\mathbf{n}$, then $\sigma$ appears in either $T_x$ or $T_y$. So, by the second part of the induction hypothesis, at least one of $x$ and $y$ is $\mathbf{n}$-expandable with respect to $\{\troot_T\}$, and, by the first part of the induction hypothesis, the other one is either $\mathbf{n}$-expandable or $\mathbf{n}$-contractible. Hence, by Proposition~\ref{prop:expandable-contractible-parent}, we get that $z$ is $\mathbf{n}$-expandable with respect to $\{\troot_T\}$.
    
    \item If a $\sigma\in\Sigma\setminus\Sigma_\mathbf{n}$ appears in $T_z$, and for any $\beta\in\Sigma_\mathbf{n}$ that appears in $T_z$ we have
    $d_T(\beta)\in\{2^{-\mathbf{n}_\beta},2^{-(\mathbf{n}_\beta-1)}\}$,
    then at least one of $x$ and $y$ satisfies the conditions for the third part of the proposition. So at least one of $x$ and $y$ is both $\mathbf{n}$-expandable and $\mathbf{n}$-super-contractible with respect to $\{\troot_T\}$ and the other one is either $\mathbf{n}$-expandable or $\mathbf{n}$-contractible by the first part of the induction hypothesis. So, by Proposition~\ref{prop:expandable-contractible-parent}, we get that $z$ is both $\mathbf{n}$-expandable and $\mathbf{n}$-super-contractible with respect to $\{\troot_T\}$.
  \end{enumerate}
  
  \vspace{0.5cm}
  {\em Case (II): either $x$ or $y$ has a label.}\\
  Assume that $x$ has a label. If $z$ has a label, the proposition follows easily from Proposition~\ref{prop:expandable-contractible-leaf} and the fact that $z$ is $\mathbf{n}$-isolated in $T$. So, assume that $z$ does not have a label.\\
  We can further assume that $x$ and $y$ are not $\mathbf{n}$-isolated in $T$, otherwise the proposition follows from case (I). So, there is a character $\sigma\in\Sigma_\mathbf{n}$ that appears in both $T_x$ and $T_y$. Since $\sigma$ appears in $T_x$ and $x$ has a label, $x$ has label $\sigma$. Note that since (i) $z$ does not have a label, (ii) $x$ has label $\sigma$, and (iii) $\sigma$ appears in $T_y$, we have $2^{-\mathbf{n}_\sigma} < d_T(\sigma) < 2^{-(\mathbf{n}_\sigma-1)}$.
  
  We will show the second part of the proposition, which will give us the first part for free. To show that $z$ is $\mathbf{n}$-expandable with respect to $\{\troot_T\}$, we need to show that $s_z$ is $(i,d)_\mathbf{n}$-relaxable for any $i\in[0,\abs{I_z}]$ and $\diff_{I_z}\leq d \leq 2\ \diff_{I_z}$. Fix $i\in[0,\abs{I_z}]$ and $\diff_{I_z}\leq d \leq 2\ \diff_{I_z}$.
  
  Let $\bar{\sigma}\neq\,\sim$ be a character not in $\Sigma$, $\bar{\Sigma} = \Sigma\cup\{\bar{\sigma}\}$, $\bar{\Sigma}_{\bar{\mathbf{n}}} = \Sigma_\mathbf{n}$ and $\bar{\mathbf{n}}\in\N^{\bar{\Sigma}_{\bar{\mathbf{n}}}}$ be equal to $\mathbf{n}$. Let $\bar{T}$ be equal to $T$ except that all the labels $\sigma$ in $\bar{T}_{\bar{y}}$ are replaced by $\bar{\sigma}$, where $\bar{v}\in \bar{T}$ is the corresponding vertex to $v\in T$ for any $v\in T$, i.e.\ $I_{\bar{v}}=I_v$.
  
  We can easily see that $\bar{z},\bar{x},\bar{y}$ are $\bar{\mathbf{n}}$-isolated in $\bar{T}$, which implies $\Sigma_{\tstring(\bar{x})}\cap\Sigma_\mathbf{n}$ and $\Sigma_{\tstring(\bar{y})}\cap\Sigma_\mathbf{n}$ are disjoint, $\bar{y}\in \bar{T}$ is $\bar{\mathbf{n}}$-expandable with respect to $\{\troot_{\bar{T}}\}$ (by the second part of the induction hypothesis for $\bar{y}$), and $\bar{x}\in \bar{T}$ is $\mathbf{n}$-contractible with respect to $\{\troot_{\bar{T}}\}$ (by Proposition ~\ref{prop:expandable-contractible-leaf}). So, by Propositoin ~\ref{prop:expandable-contractible-parent}, $\bar{z}\in \bar{T}$ is $\bar{\mathbf{n}}$-expandable with respect to $\{\troot_{\bar{T}}\}$, which implies that $\tstring(\bar{z})$ is $(i,d)_{\bar{\mathbf{n}}}$-relaxable. So, there are $k\in\N, \bar{u}_z\in\bar{\Sigma}^{\END}, \bar{t}_z\in\bar{\Sigma}^{\DOM}$, where $\DOM =\DOM_{\bar{\mathbf{n}}}(I_{\bar{z}}, i, k, d) = \DOM_{\mathbf{n}}(I_z, i, k, d)$ and $\STR,\END$ are defined similarly, such that
  \begin{itemize}
    \item $\stree(\bar{u}_z) \cong \stree(\tstring(\bar{z})) = \bar{T}_{\bar{z}}$, and
    \item $\tstring(\bar{z})\to_{\bar{\mathbf{n}}, \bar{t}_z} \bar{u}_z$,
  \end{itemize}
  
  Let $t_z$ and $u_z$ be equal to $\bar{t}_z$ and $\bar{u}_z$ except that all the $\bar{\sigma}$'s are replaced by $\sigma$. $\stree(u_z)$ is equal to $\stree(\bar{u}_z)$ when all labels $\bar{\sigma}$ are replaced by $\sigma$, which is isomorphic to $\stree(\tstring(\bar{z}))$ when all labels $\bar{\sigma}$ are replaced by $\sigma$, which is equal to $\stree(s_z) = T_z$. So, $\stree(u_z)\cong T_z$. We also have the following which shows $s_z\to_{\mathbf{n}, t_z} u_z$.
  \begin{itemize}
    \item
      $
      \Sigma_{t_z} = \Sigma_{\bar{t}_z} \setminus \{\bar{\sigma}\} \cup \{\sigma\}
      = \Sigma_{\tstring(\bar{z})} \setminus \{\bar{\sigma}\} \cup \{\sigma\}
      = \Sigma_{s_z}
      $.\\
      Similarly, we get $\Sigma_{u_z} = \Sigma_{s_z}$.
    
    \item Since $d\geq \diff_{I_z}$, all the gaps between indices of $\DOM$ are at least $\diff_{I_z}$. On the other hand, all the gaps between appearances of $\sigma$ in $\bar{t}_z$ are at most $2^{\mathbf{n}_\sigma}=2\ \diff_{I_z}$. So, in $t_z$ all the gaps between appearances of $\sigma$ are between $\diff_{I_z}$ and $2\ \diff_{I_z}$. The gaps for the other characters of $\Sigma_{t_z}\cap\Sigma_\mathbf{n}$ are the same in $t_z$ and $\bar{t}_z$. So, $t_z$ is locally $\mathbf{n}$-compatible.
    
    \item $t_z[\mathbf{n}, +1]$ is equal to $\bar{t}_z[\mathbf{n}, +1] = \tstring(\bar{z})$ when all $\bar{\sigma}$'s are replaced by $\sigma$, which is equal to $s_z$.
    \item $t_z[\mathbf{n}, -1]$ is equal to $\bar{t}_z[\mathbf{n}, -1] = \bar{u}_z$ when all $\bar{\sigma}$'s are replaced by $\sigma$, which is equal to $u_z$.
  \end{itemize}
  This shows that $s_z$ is $(i,d)_\mathbf{n}$-relaxable. So, $z\in T$ is $\mathbf{n}$-expandable with respect to $\{\troot_T\}$. Note that in this case $z$ cannot satisfy the conditions for the third part of the proposition.

\end{proof}

\begin{definition}[Tree connections]
\label{def:adjacent-trees}
  Let $\Sigma$ be a finite alphabet and $\mathbf{n}\in \N^{\Sigma_\mathbf{n}}$ be a frame on $\Sigma$.
  \begin{enumerate}
    \item Let $s, s'\in \Sigma^{[2^{M_\mathbf{n}}]}$ be $\mathbf{n}$-compatible strings. We write $s\to_\mathbf{n} s'$ if there are $k\in\N, t_{s,s'}\in \Sigma^{[k\,2^{M_\mathbf{n}}]}$ such that $s\to_{\mathbf{n}, t_{s,s'}} s'$.
    \item Let $T$ and $T'$ be $\mathbf{n}$-compatible trees and $s=\tstring(T)$ and $s'=\tstring(T')$, we write $T\to_\mathbf{n} T'$ iff $s\to_\mathbf{n} s'$.
  \end{enumerate}
  In this case
  \begin{enumerate}
    \item Any such $t_{s,s'}$ is called a {\em connecting string} for $(T,T')_\mathbf{n}$.
    \item $\IND_\mathbf{n}(t_{s,s'}) \coloneqq \{i\in\domain(t_{s,s'})\;|\; t_{s,s'}(i)\neq t_{s,s'}[\mathbf{n},+1]^{\wedge_\mathbf{n}k}(i)\}$. We call $\IND_\mathbf{n}(t_{s,s'})$ the active indices for $(t_{s,s'})_\mathbf{n}$.\\
    Note that since $t_{s,s'}[\mathbf{n},+1] = s$, active indices for $(t_{s,s'})_\mathbf{n}$ is in fact equal to $\{i\in\domain(t_{s,s'})\;|\; s^{\wedge_\mathbf{n}k}(i)\neq t_{s,s'}(i)\}$
    \item For any vertex $v\in T$ define\\
    $\Sigma^{\text{unique}}_v = \{\sigma\in\Sigma_{\tstring(v)}\ | \ \sigma \text{ does not appear in } T\setminus T_v\}$, and\\
    $\IND_\mathbf{n}(t_{s,s'},v)\coloneqq \{i\in\domain(t_{s,s'})\;|\; t_{s,s'}(i)\in \Sigma^{\text{unique}}_v\}$.
  \end{enumerate}
\end{definition}

\begin{definition}[Permutations]
\label{def:permutation}
  Let $\Sigma$ be a finite alphabet, $\mathbf{n}\in \N^{\Sigma_\mathbf{n}}$ be a frame on $\Sigma$, and $T$ be an $\mathbf{n}$-compatible tree. Let $v_1,\ldots,v_m$ be $\mathbf{n}$-isolated vertices in $T$ with the same depth, and let $\pi$ be a permutation for $V=\{v_1,\ldots,v_m\}$. We say that we can apply $\pi$ to $(T,v_1,\ldots,v_m)_\mathbf{n}$ if there exists an $\mathbf{n}$-compatible tree $T'$ such that
  \begin{itemize}
    \item $T'_{v'}\cong T_{\pi^{-1}(v)}$ for $v\in V$, where $w'\in T'$ is the corresponding vertex to $w$ for any $w\in T$, i.e.\ $I_{w'} = I_w$, and
    \item $T\to_\mathbf{n} T'$ and for some connecting string $t_{T,T'}$ for $(T,T')_\mathbf{n}$ we have $\IND_\mathbf{n}(t_{T,T'}) \subseteq \cup_{v\in V} I_v \mod 2^{M_\mathbf{n}}$ and $\IND_\mathbf{n}(t_{T,T'}, v) \subseteq I_v\cup I_{\pi(v)} \mod 2^{M_\mathbf{n}}$ for all $v\in V$.
  \end{itemize}
  In this case, any such $T'$ is called a {\em result of $\pi$ on $(T,v_1,\ldots,v_m)_\mathbf{n}$}.
\end{definition}

Note that if $T$ is an $(\mathbf{n}, \mathbf{p})$-uniform tree and $T'$ is a result of a permutation $\pi$ on $(T,v_1,\ldots,v_m)_\mathbf{n}$, then $T'$ is also an $(\mathbf{n}, \mathbf{p})$-uniform tree.

\begin{lemma}
\label{lem:perm}
  Let $\Sigma$ be a finite alphabet, $\mathbf{n}\in \N^{\Sigma_\mathbf{n}}$ be a frame on $\Sigma$, and $T$ be an $(\mathbf{n},d_T)$-uniform tree. Let $v_1,\ldots,v_m$ be $\mathbf{n}$-isolated vertices in $T$ with the same depth, and let $\pi$ be a permutation of $V=\{v_1,\ldots,v_m\}$.\\
  Assume for $v\in V$ there are $i_v\in[0,\abs{I_v}]$ and $d_v\in\N$ such that
  \begin{itemize}
    \item $\tstring(v)$ is $(i_v,d_v)_\mathbf{n}$-relaxable for all $v\in V$,
    \item $d_v+I_v = I_{\pi(v)} \mod 2^{M_\mathbf{n}}$ for all $v\in V$, and
    \item $A_{v_1},\ldots,A_{v_m},A'$ is a partition of $A$, where
    \begin{align*}
      & A_v = \DOM_\mathbf{n}(I_v, i_v, 1, d_v) \text{ for } v\in V,\\
      & A = \DOM_\mathbf{n}([2^{M_\mathbf{n}}], 0, 1, 1) = [4\cdot 2^{M_\mathbf{n}}], \text{ and }\\
      & A' = \{l\in A\;|\; l\notin \cup_{v\in V} I_v \mod 2^{M_\mathbf{n}}\}.
    \end{align*}
  \end{itemize}
  Then we can apply $\pi$ to $(T,v_1,\ldots,v_m)_\mathbf{n}$.
\end{lemma}

\begin{proof}
  Let $s_v = \tstring(v)$ for $v\in V$ and let $s = \tstring(T)$.
  
  For $v\in V$, since $s_v$ is $(i_v,d_v)_\mathbf{n}$-relaxable, there are $k_v\in\N$, $u_v\in\Sigma^{\END_v}$, $t_v\in\Sigma^{\DOM_v}$, where $\DOM_v = \DOM_\mathbf{n}(I_v, i_v, k_v, d_v)$ and $\END_v$ is defined similarly, such that $s_v\to_{\mathbf{n},t_v}u_v$ and $\stree(u_v)\cong \stree(s_v) = T_v$. Note that $\END_v = I_{\pi(v)}$.
  
  Without loss of generality, we may assume $k_{v_1} = \cdots = k_{v_m}$. The reason for this is that increasing $k_v$ does not affect existence of such $u_v$ and $t_v$: if $u_v$ and $t_v$ work for $k_v$, then $u_v$ and $s_v\wedge_\mathbf{n} t_v$ work for $k_v+1$. So, let $k \coloneqq k_{v_1} = \cdots = k_{v_m}$.
  
  Let 
  \begin{align*}
    &\DOM = [(k+3)2^{M_\mathbf{n}}],\\
    &\STR = \END = [2^{M_\mathbf{n}}],\\
    &\DOM'=\{l\in\DOM\ | \ l \notin \cup_{v\in V} I_v \mod 2^{M_\mathbf{n}}\},\\
    &\STR' = \END' = [2^{M_\mathbf{n}}]\setminus(\cup_{v\in V} I_v).
  \end{align*}

  Let $s'=u'\in\Sigma^{\STR'}$ be equal to the restriction of $s$ to $\STR'$, and $t' = (s')^{\wedge_\mathbf{n}(k+3)}\in \Sigma^{\DOM'}$. Since $T$ is $(\mathbf{n},d_T)$-uniform and $v_1,\ldots,v_k$ are $\mathbf{n}$-isolated, we can easily see that $s'\to_{\mathbf{n},t'}u'$.
    
  So, we have $s_v\to_{\mathbf{n}, t_v}u_v$ for $v\in V$ and $s'\to_{\mathbf{n},t'}s'$, and
  \begin{enumerate}
    \item $\max_\mathbf{n}\DOM_{v_1} = \cdots = \max_\mathbf{n}\DOM_{v_m} = \max_\mathbf{n}\DOM' = k+3$,
    \item since $A_{v_1},\ldots,A_{v_m},A'$ is a partition of $A$, we have that 
    $$\DOM_{v_1},\ldots,\DOM_{v_m},\DOM'$$
    is a partition of $\DOM$, and the same holds for $\STR$ and $\END$, and 
    \item since $T$ is $(\mathbf{n},d_T)$-uniform and $v_1,\ldots,v_m$ are $\mathbf{n}$-isolated,
    $$\Sigma_{s_{v_1}}\cap\Sigma_\mathbf{n}, \ldots, \Sigma_{s_{v_m}}\cap\Sigma_\mathbf{n}, \Sigma_{s'}\cap\Sigma_\mathbf{n}$$
    are disjoint.
  \end{enumerate}
  So, by Lemma~\ref{lem:string-graph} part (2), and since $s = s_{v_1}\vee\cdots\vee s_{v_m}\vee s'$, we get
  \begin{equation}
  \label{eq:perm_graph}
    s\to_{\mathbf{n},t} u,
  \end{equation}
  where $t\coloneqq t_{v_1}\vee\cdots\vee t_{v_m}\vee t'$ and $u\coloneqq u_{v_1}\vee\cdots\vee u_{v_m}\vee u'$. Since $\DOM_{v_1},\ldots,\DOM_{v_m},\DOM'$ is a partition of $\DOM$, we have $t\in\Sigma^{\DOM}$ and $u\in\Sigma^{\END}$.

  Let $T' = \stree(u)$. For any vertex $w\in T$, let $w'$ be its corresponding vertex in $T'$. For $v\in V$, we know that $\tstring(T'_{v'})$ is the restriction of $u$ to $I_v$, and since $I_v = \END_{\pi^{-1}(v)}$, this is equal to the restriction of $u$ to $\END_{\pi^{-1}(v)}$, which is equal to $u_{\pi^{-1}(v)}$. So,
  \begin{equation}
  \label{eq:perm_tree}
    T'_{v'} = \stree(u_{\pi^{-1}(v)}) \cong T_{\pi^{-1}(v)}.
  \end{equation}
  
  By ~\eqref{eq:perm_graph} and ~\eqref{eq:perm_tree}, to complete the proof of this lemma, we just need to show that
  \begin{itemize}
    \item $\IND_\mathbf{n}(t) \subseteq \cup_{v\in V} I_v \mod 2^{M_\mathbf{n}}$, and
    \item $\IND_\mathbf{n}(t, v) \subseteq I_v\cup I_{\pi(v)} \mod 2^{M_\mathbf{n}}$ for $v\in V$.
  \end{itemize}
  
  If $l\notin \cup_{v\in V} I_v\mod 2^{M_\mathbf{n}}$, we have $t(l) = t'(l) = s(l')$ for $l'\in [2^{M_\mathbf{n}}]$ with $l' = l \mod 2^{M_\mathbf{n}}$. So, the first statement holds.
  
  To show the second statement, fix $v\in V$. Let $\sigma\in\Sigma^{\text{uniqe}}_v$. By definition, $\sigma\notin \Sigma_{s_w}$ for all $w\in V\setminus\{v\}$ and also $\sigma\notin \Sigma_{s'}$. Since $\Sigma_{t_w} = \Sigma_{s_w}$ and $\Sigma_{t'} = \Sigma_{s'}$, we also get that $\sigma\notin \Sigma_{t_w}$ for all $w\in V\setminus\{v\}$ and also $\sigma\notin \Sigma_{t'}$. This shows that $t(l) = \sigma$ only if $l\in\DOM_v$. But $\DOM_v\subseteq I_v\cup(d_v + I_v) = I_v\cup I_{\pi(v)}\mod 2^{M_\mathbf{n}}$. This completes the proof.
\end{proof}

\begin{proposition}
\label{prop:2-swap}
  Let $\Sigma$ be a finite alphabet, $\mathbf{n}\in \N^{\Sigma_\mathbf{n}}$ be a frame on $\Sigma$, and $T$ be an $(\mathbf{n},d_T)$-uniform tree.
  
  Let $x, y\in T$ have the same depth. Assume that $x,y$ are $\mathbf{n}$-isolated in $T$, $x$ is $\mathbf{n}$-expandable with respect to $\{x,y\}$, $y$ is $\mathbf{n}$-contractible with respect to $\{x,y\}$, and $d(x\to y) \leq 1/2\ \diff_{I_x}$. Then we can swap $x$ and $y$, meaning that we can apply the permutation $\pi$ with $\pi(x)=y, \pi(y)=x$ to $(T,x,y)_\mathbf{n}$.
\end{proposition}

\begin{proof}
  Let
  \begin{align*}
    & d_x = \diff_{I_x} + d(x\to y), & & i_x = 0,\\
    & d_y = d(y\to x), & & i_y = \delta_{\min I_y < \min I_x}.
  \end{align*}
  
  \begin{itemize}
    \item Note that $\diff_{I_x} \leq d_x \leq 2 \diff_{I_x}$ and $d_x + I_x \subseteq I_{\{x,y\}} \mod 2^{M_\mathbf{n}}$. So, since $x$ is $\mathbf{n}$-expandable with respect to $\{x,y\}$, $s_x$ is $(i_x, d_x)_\mathbf{n}$-relaxable. Similarly, we can see that $s_y$ is $(i_y, d_y)_\mathbf{n}$-relaxable.
    
    \item Note that $d_x + I_x = I_y, d_y + I_y = I_x \mod 2^{M_\mathbf{n}}$.
    
    \item Let $A_x = \DOM_\mathbf{n}(I_x, i_x, 1, d_x), A_y = \DOM_\mathbf{n}(I_y, i_y, 1, d_y), A = [4\cdot 2^{M_\mathbf{n}}], A' = \{l\in A\;|\; l\notin I_x\cup I_y \mod 2^{M_\mathbf{n}}\}$. It is straightforward to see that $A_x,A_y,A'$ is a partition of $A$.
  \end{itemize}
  So, Lemma~\ref{lem:perm} shows that we can apply the permutation $\pi$ with $\pi(x)=y,\pi(y)=x$ to $(T,x,y)_\mathbf{n}$.
\end{proof}

\begin{proposition}
\label{prop:4-perm}
  Let $\Sigma$ be a finite alphabet, $\mathbf{n}\in \N^{\Sigma_\mathbf{n}}$ be a frame on $\Sigma$, and $T$ be an $(\mathbf{n},d_T)$-uniform tree.
  
  Let $x\neq y \in T$ have the same depth, $x_1, x_2$ be the two children of $x$, and $y_1, y_2$ be the two children of $y$. Assume that $x_1,x_2,y_1,y_2$ are $\mathbf{n}$-isolated in $T$, $x_1$ is $\mathbf{n}$-expandable with respect to $\{x,y\}$, $y_1$ is $\mathbf{n}$-contractible with respect to $\{y\}$, and $d(x_1\to y_1)\leq \diff_{I_{x_1}} / 2$. Then there exists a permutation $\pi$ of $\{x_1,x_2,y_1,y_2\}$ with $\pi(x_1),\pi(y_1)\in\{y_1,y_2\}$ that can be applied to $(T,x_1,x_2,y_1,y_2)_\mathbf{n}$.
\end{proposition}

\begin{proof}
  Since $x_2, y_2$ are $\mathbf{n}$-isolated in $T$, by Proposition~\ref{prop:expandable-contractible-isolated} part (1), we know that each of $x_2$ and $y_2$ is either $\mathbf{n}$-expandable or $\mathbf{n}$-contractible with respect to $\{\troot_T\}$. We prove this proposition by considering the following cases. In each case, we define a permutation $\pi$ of $\{x_1,x_2,y_1,y_2\}$ with $\pi(x_1),\pi(y_1)\in\{y_1,y_2\}$ and for each $v\in\{x_1,x_2,y_1,y_2\}$ we define $i_v\in[0,\abs{I_v}]$ and $d_v\in\N$. In each case, similar to the proof of Proposition~\ref{prop:2-swap}, it is straightforward to see that the conditions for Lemma~\ref{lem:perm} hold, and hence $\pi$ can be applied to $(T,x_1,x_2,y_1,y_2)_\mathbf{n}$. Let $d = \diff_{I_x} = \diff_{I_y}$. 
  
  \vspace{0.3cm}
  {\em Case (I): $x_2$ is $\mathbf{n}$-expandable, $y_2$ is $\mathbf{n}$-contractible.}\\
  \begin{align*}
    & d_{x_1} = 2\,d + d(x_1\to y_1), & & i_{x_1} = 1,\\
    & d_{x_2} = 2\,d + d(x_2\to x_1), & & i_{x_2} = \delta_{\min I_{x_2} < \min I_{x_1}},\\
    & d_{y_1} = d(y_1\to y_2), & & i_{y_1} = \delta_{\min I_{y_1} < \min I_{x_1}} + 1,\\
    & d_{y_2} = d(y_2\to x_2), & & i_{y_2} = \delta_{\min I_{y_2} < \min I_{x_1}},\\
    & \pi(x_1)=y_1, \pi(x_2)=x_1, \pi(y_1)=y_2, \pi(y_2)=x_2.
  \end{align*}
  
  \vspace{0.3cm}
  {\em Case (II): $x_2$ is $\mathbf{n}$-expandable, $y_2$ is $\mathbf{n}$-expandable.}\\
  \begin{align*}
    & d_{x_1} = 2\,d + d(x_1\to y_1), & & i_{x_1} = 1,\\
    & d_{x_2} = 2\,d, & & i_{x_2} = 1,\\
    & d_{y_1} = d(y_1\to y_2), & & i_{y_1} = \delta_{\min I_{y_1} < \min I_{x_1}} + 1,\\
    & d_{y_2} = 2\,d + d(y_2\to x_1), & & i_{y_2} = \delta_{\min I_{y_2} < \min I_{x_1}},\\
    & \pi(x_1)=y_1, \pi(x_2)=x_2, \pi(y_1)=y_2, \pi(y_2)=x_1.
  \end{align*}

  \vspace{0.3cm}
  {\em Case (III): $x_2$ is $\mathbf{n}$-contractible, $y_2$ is $\mathbf{n}$-contractible.}\\
  \begin{align*}
    & d_{x_1} = 2\,d + d(x_1\to y_2), & & i_{x_1} = 0,\\
    & d_{x_2} = d(x_2\to x_1), & & i_{x_2} = \delta_{\min I_{x_2} < \min I_{x_1}},\\
    & d_{y_1} = 2\,d, & & i_{y_1} = 0,\\
    & d_{y_2} = d(y_2\to x_2), & & i_{y_2} = \delta_{\min I_{y_2} < \min I_{x_1}},\\
    & \pi(x_1)=y_2, \pi(x_2)=x_1, \pi(y_1)=y_1, \pi(y_2)=x_2.
  \end{align*}

  \vspace{0.3cm}
  {\em Case (IV): $x_2$ is $\mathbf{n}$-contractible, $y_2$ is $\mathbf{n}$-expandable.}\\
  \begin{align*}
    & d_{x_1} = 2\,d + d(x_1\to y_1), & & i_{x_1} = 1,\\
    & d_{x_2} = d(x_2\to x_1), & & i_{x_2} = \delta_{\min I_{x_2} < \min I_{x_1}} + 1,\\
    & d_{y_1} = d(y_1\to y_2), & & i_{y_1} = \delta_{\min I_{y_1} < \min I_{x_1}} + 1,\\
    & d_{y_2} = 2\,d + d(y_2\to x_2), & & i_{y_2} = \delta_{\min I_{y_2} < \min I_{x_1}},\\
    & \pi(x_1)=y_1, \pi(x_2)=x_1, \pi(y_1)=y_2, \pi(y_2)=x_2.
  \end{align*}

\end{proof}

\begin{proposition}
\label{prop:4-swap}
  Let $\Sigma$ be a finite alphabet, $\mathbf{n}\in \N^{\Sigma_\mathbf{n}}$ be a frame on $\Sigma$, and $T$ be an $(\mathbf{n},d_T)$-uniform tree.
  
  Let $x\neq y \in T$ have the same depth, $x_1, x_2$ be the two children of $x$, and $y_1, y_2$ be the two children of $y$. Assume that $x_1,x_2,y_1,y_2$ are $\mathbf{n}$-isolated in $T$, $x_1$ is $\mathbf{n}$-expandable with respect to $\{x,y\}$, and $y_1$ is $\mathbf{n}$-contractible with respect to $\{y\}$.

   Then there exists a permutation $\pi$ of $V = \{x_1,x_2,y_1,y_2\}$ with $\pi(x_1),\pi(y_1)\in\{y_1,y_2\}$ and an $\mathbf{n}$-compatible tree $T'$ such that
  \begin{itemize}
    \item $T'_{v'}\cong T_{\pi^{-1}(v)}$ for $v\in V$, where $w'\in T'$ is the corresponding vertex to $w$ for any $w\in T$, and
    \item $T\to_\mathbf{n} T'$ and for some connecting string $t_{T,T'}$ for $(T,T')_\mathbf{n}$ we have $\IND_\mathbf{n}(t_{T,T'}) \subseteq \cup_{v\in V} I_v \mod 2^{M_\mathbf{n}}$ and $\IND_\mathbf{n}(t_{T,T'}, y_1) \subseteq I_y \mod 2^{M_\mathbf{n}}$.
  \end{itemize}
\end{proposition}

\begin{proof}
  Note that if there exists a permutation $\pi$ of $V$ with $\pi(x_1),\pi(y_1)\in\{y_1,y_2\}$ that can be applied to $(T,x_1,x_2,y_1,y_2)_\mathbf{n}$, then by Definition~\ref{def:permutation} we are done.
  
  Let $d = \diff_{I_x} = \diff_{I_y}$. If $d(x_1\to y_1)\leq d$, the result follows from Proposition~\ref{prop:4-perm}. So, assume $d(x_1\to y_1) > d$, which implies $d(x_1\to y_2)\leq d$.
  
  Since $y_2$ is $\mathbf{n}$-isolated in $T$, by the first part of Proposition~\ref{prop:expandable-contractible-isolated} we know that $y_2$ is either $\mathbf{n}$-expandable or $\mathbf{n}$-contractible with respect to $\{\troot_T\}$. Consider the following cases.
  
  \vspace{0.3cm}
  {\em Case (I): $y_2$ is $\mathbf{n}$-contractible with respect to $\{\troot_T\}$.}\\
  Since $x_1,y_2$ are $\mathbf{n}$-isolated in $T$, $x_1$ is $\mathbf{n}$-expandable with respect to $\{x,y\}$ and hence with respect to $\{x_1,y_2\}$, $y_2$ is $\mathbf{n}$-contractible with respect to $\{\troot_T\}$ and hence with respect to $\{x_1,y_2\}$, and $d(x_1\to y_2)\leq d = \diff_{I_{x_1}} / 2$, by Proposition ~\ref{prop:2-swap}, we can swap $x_1$ and $y_2$, which means that we can apply the permutation $\pi$ of $V$ with $\pi(x_1)=y_2,\pi(x_2)=x_2,\pi(y_1)=y_1,\pi(y_2)=x_1$ to $(T,x_1,x_2,y_1,y_2)_\mathbf{n}$. That completes the proof in this case.
  
  \vspace{0.3cm}
  {\em Case (II): $y_2$ is $\mathbf{n}$-expandable with respect to $\{\troot_T\}$.}\\
  Since $y_1,y_2$ are $\mathbf{n}$-isolated in $T$, $y_1$ is $\mathbf{n}$-contractible with respect to $\{x,y\}$ and hence with respect to $\{y_1,y_2\}$, $y_2$ is $\mathbf{n}$-expandable with respect to $\{\troot_T\}$ and hence with respect to $\{y_1,y_2\}$, and $d(y_1\to y_2) = d = \diff_{I_{y_1}} / 2$, by Proposition ~\ref{prop:2-swap}, we can swap $y_1$ and $y_2$, which means that we can apply the permutation $\phi$ of $V$ with $\phi(x_1)=x_1,\phi(x_2)=x_2,\phi(y_1)=y_2,\phi(y_2)=y_1$ to $(T,x_1,x_2,y_1,y_2)_\mathbf{n}$. So, there exists an $(\mathbf{n},d_T)$-uniform tree $T''$ such that if for any $w\in T$ we denote its corresponding vertex in $T''$ by $w''$, i.e.\ $I_{w''} = I_w$, then we have
  \begin{itemize}
    \item $T''_{v''}\cong T_{\phi^{-1}(v)}$ for $v\in V$, and
    \item $T\to_\mathbf{n} T''$ and for some connecting string $t_{T,T''}$ for $(T,T'')_\mathbf{n}$, we have $\IND_\mathbf{n}(t_{T,T''}) \subseteq I_x\cup I_y \mod 2^{M_\mathbf{n}}$ and $\IND_\mathbf{n}(t_{T,T''},y_1)\subseteq I_{y_1}\cup I_{\phi(y_1)} = I_{y_1}\cup I_{y_2} = I_y \mod 2^{M_\mathbf{n}}$.
  \end{itemize}
  
  Let $V'' = \{x_1'', x_2'', y_1'', y_2''\}$. Note that all $v''\in V''$ are $\mathbf{n}$-isolated in $T''$. It is straightforward to see that since $y_1$ is $\mathbf{n}$-contractible with respect to $\{x,y\}$, $y_2''$ is $\mathbf{n}$-contractible with respect to $\{x'',y''\}$. Similarly, $x_1''$ is $\mathbf{n}$-expandable with respect to $\{x'',y''\}$. Also, $d(x_1''\to y_2'') = d(x_1\to y_2) \leq d = \diff_{I_{x_1''}} / 2$. So, by Proposition ~\ref{prop:4-perm}, there exists a permutation $\phi''$ of $V''$ with $\phi''(x_1''),\phi''(y_2'')\in\{y_1'',y_2''\}$ that can be applied to $(T'',x_1'',x_2'',y_1'',y_2'')_\mathbf{n}$. So, there exists an $(\mathbf{n},d_T)$-uniform tree $T'$ such that if for any $w\in T$ we denote its corresponding vertex in $T'$ by $w'$, i.e.\ $I_{w'} = I_w$, then we have
  \begin{itemize}
    \item $T'_{v'}\cong T''_{\phi''^{-1}(v'')}$ for $v\in V$, and
    \item $T''\to_\mathbf{n} T'$ and for some connecting string $t_{T'',T'}$ for $(T'',T')_\mathbf{n}$, we have $\IND_\mathbf{n}(t_{T'',T'}) \subseteq I_x\cup I_y \mod 2^{M_\mathbf{n}}$ and $\IND_\mathbf{n}(t_{T'',T'},y_2'')\subseteq I_{y_2''}\cup I_{\phi''(y_2'')} = I_{y_2''}\cup I_{y_1''} = I_y \mod 2^{M_\mathbf{n}}$.
  \end{itemize}
  
  Let $\alpha:T''\to T$ be defined by $\alpha(v'') = v$ for any $v\in T$. Define $\pi:V\to V$ by
  $$\pi(v) = \alpha\big(\phi''\big(\phi(v)''\big)\big) $$
  for any $v\in V$. Note that $\pi$ is a permutation of $V$ and from the results in the previous two paragraphs, we get $T'_{v'}\cong T_{\pi^{-1}(v)}$ for $v\in V$. Let $t_{T,T'} = t_{T,T''}\diamond_\mathbf{n} t_{T'',T'}$. By Lemma~\ref{lem:string-graph} part (1), we get that $T\to_{\mathbf{n}, t_{T,T'}} T'$. To complete the proof of this proposition, we need to show the following.
  
  \begin{itemize}
  \item First, we need to show $\IND_\mathbf{n}(t_{T,T'}) \subseteq I_x\cup I_y \mod 2^{M_\mathbf{n}}$. Since
  $$\IND_\mathbf{n}(t_{T,T'}) \subseteq \IND_\mathbf{n}(t_{T,T''}) \cup \IND_\mathbf{n}(t_{T'',T'}) \mod 2^{M_\mathbf{n}},$$
  this follows from $\IND_\mathbf{n}(t_{T,T''}), \IND_\mathbf{n}(t_{T'',T'}) \subseteq I_x\cup I_y \mod 2^{M_\mathbf{n}}$.
  
  \item Then, we need to show $\IND_\mathbf{n}(t_{T,T'},y_1)\subseteq I_y \mod 2^{M_\mathbf{n}}$. Note that mod $2^{M_\mathbf{n}}$ we have
  \begin{align*}
    \IND_\mathbf{n}(t_{T,T'},y_1) = & \{i\in\domain(t_{T,T'})\;|\; t_{T,T'}(i)\in \Sigma^{\text{uniqe}}_{\tstring(y_1)}\}\\
    = & \{i\in\domain(t_{T,T''})\;|\; t_{T,T''}(i)\in \Sigma^{\text{uniqe}}_{\tstring(y_1)} \}\\
    & \cup \{i\in\domain(t_{T'',T'})\;|\; t_{T'',T'}(i)\in \Sigma^{\text{uniqe}}_{\tstring(y_2'')} \}\\
    = & \IND_\mathbf{n}(t_{T,T''},y_1) \cup \IND_\mathbf{n}(t_{T'',T'},y_2'')\\
    \subseteq & I_y \mod 2^{M_\mathbf{n}}.
  \end{align*}
  \end{itemize}
  
  So, the proof is complete.
\end{proof}

\subsection{Proof of the Second Claim in Lemma ~\ref{lem:n-p-good-strings}}
Let $\Sigma$ be a finite alphabet, $\mathbf{n}\in \N^\Sigma$ be a frame on $\Sigma$, $\mathbf{p}, \mathbf{p}'\in \PBA_\mathbf{n}$, and $T$ be an $(\mathbf{n}, \mathbf{p})$-uniform tree. We want to show there exists an $(\mathbf{n}, \mathbf{p}')$-uniform tree $T'$ such that $T\to_\mathbf{n} T'$.

Since $T^{(1)}\to_\mathbf{n} T^{(2)}$ and $T^{(2)}\to_\mathbf{n} T^{(3)}$ implies $T^{(1)}\to_\mathbf{n} T^{(3)}$, we will prove Lemma~\ref{lem:n-p-good-strings} part (2) by a series of reductions.

\subsubsection{Step 1}
Note that $\PBA_\mathbf{n}$ is finite. Define a graph, $\Gamma_\mathbf{n}$, on the vertex set $\PBA_\mathbf{n}$ by connecting two pseudo-binary approximations $\mathbf{p}$, $\mathbf{p}'$ if and only if they differ on exactly two characters, say $\alpha$ and $\beta$, and also $\irr(\mathbf{p}), \irr(\mathbf{p}') \in \{\sim, \alpha, \beta\}$.

\begin{claim}
\label{claim:Gamma-connected}
  $\Gamma_\mathbf{n}$ is connected.
\end{claim}
\begin{proof}
  For $\mathbf{p},\mathbf{p}'\in\PBA_\mathbf{n}$, let 
  \begin{align*}
  \Sigma_{\mathbf{p}, \mathbf{p}'} & = \{\sigma\in\Sigma \ | \ \mathbf{p}_\sigma\neq \mathbf{p}'_\sigma\} \cup \{\irr(\mathbf{p}), \irr(\mathbf{p}')\} \setminus \{\sim\},\\
  d_{\mathbf{p},\mathbf{p}'} 
  & = \sum_{\sigma\in\Sigma \setminus \{\irr(\mathbf{p}),\irr(\mathbf{p}')\}} \abs{\mathbf{p}_\sigma - \mathbf{p}'_\sigma} \\
  & = \sum_{\sigma\in\Sigma} \abs{\mathbf{p}_\sigma-\mathbf{p}'_\sigma} \delta_{\sigma\notin \{\irr(\mathbf{p}),\irr(\mathbf{p}')\}}.
  \end{align*}

  Given $\mathbf{p}\neq \mathbf{p}'\in\PBA_\mathbf{n}$, if $\mathbf{p}$ and $\mathbf{p}'$ are not connected in $\Gamma_\mathbf{n}$, we define $\mathbf{p}''\in\PBA_\mathbf{n}$ such that $\mathbf{p}''$ is connected to $\mathbf{p}$ in $\Gamma_\mathbf{n}$ and either $d_{\mathbf{p},\mathbf{p}'} > d_{\mathbf{p}'',\mathbf{p}'}$, or $d_{\mathbf{p},\mathbf{p}'} = d_{\mathbf{p}'',\mathbf{p}'}$ and $\abs{\Sigma_{\mathbf{p},\mathbf{p}'}} > \abs{\Sigma_{\mathbf{p}'',\mathbf{p}'}}$. Since $\PBA_\mathbf{n}$ is finite, this shows that $\Gamma_\mathbf{n}$ is connected.
  
  Let $\mathbf{p}\neq \mathbf{p}'\in\PBA_\mathbf{n}$ be not connected in $\Gamma_\mathbf{n}$. Define $f\colon\Sigma\to [0,1]$ by
  $$
  f(\sigma) = 
  \begin{cases}
    \frac{1}{2^{\mathbf{n}_\sigma}} & \text{ if } \mathbf{p}_\sigma > \mathbf{p}'_\sigma,\\
    \frac{1}{2^{\mathbf{n}_\sigma-1}} & \text{ if } \mathbf{p}_\sigma < \mathbf{p}'_\sigma,\\
    \frac{1}{2^{\mathbf{n}_\sigma-1}} & \text{ if } \mathbf{p}_\sigma = \mathbf{p}'_\sigma \text{ and } \sigma = \irr(\mathbf{p}) = \irr(\mathbf{p}'),\\
    \mathbf{p}_\sigma & \text{ otherwise.}
  \end{cases}
  $$
  Let $A = \{\sigma\in\Sigma \ | \ f(\sigma) < \mathbf{p}_\sigma\}$ and $B = \{\sigma\in\Sigma \ | \ f(\sigma) > \mathbf{p}_\sigma\}$. Since $\mathbf{p}\neq \mathbf{p}'$, we know that $A,B\neq \emptyset$. Let $\alpha\in A \cup B$ be a character such that either $\irr(\mathbf{p}) =\ \sim$ or $\alpha = \irr(\mathbf{p})$. Assume that $\alpha\in A$ (if $\alpha\in B$, the proof is the same). Let $\beta\in B$ be a character such that either $B = \{\irr(\mathbf{p}')\}$ or $\beta \neq \irr(\mathbf{p}')$.
  
  Let
  $$\eps = \min(\mathbf{p}_\alpha - f(\alpha), f(\beta) - \mathbf{p}_\beta),$$
  and define the probability distribution $\mathbf{p}''$ on $\Sigma$ by
  $$
	\mathbf{p}''_\sigma =
	\begin{cases}
	  \mathbf{p}_\sigma & \text{ if } \sigma\in\Sigma\setminus \{ \alpha, \beta \},\\
	  \mathbf{p}_\sigma + \eps & \text{ if } \sigma = \beta,\\
	  \mathbf{p}_\sigma - \eps & \text{ if } \sigma = \alpha.
	\end{cases}
  $$
  It is straightforward to see that $\mathbf{p}''\in\PBA_\mathbf{n}$ and that $\mathbf{p}$ and $\mathbf{p}''$ are connected in $\Gamma_\mathbf{n}$. To complete the proof, it is enough to show that either $d_{\mathbf{p}'',\mathbf{p}'} \leq d_{\mathbf{p},\mathbf{p}'} - \eps$, or $d_{\mathbf{p}'',\mathbf{p}'} = d_{\mathbf{p},\mathbf{p}'}$ and $\abs{\Sigma_{\mathbf{p},\mathbf{p}'}} > \abs{\Sigma_{\mathbf{p}'',\mathbf{p}'}}$.
  
  First, assume that $\irr(\mathbf{p}) =\ \sim$. Note that we have one of the following two cases.
  \begin{itemize}
    \item $\beta \neq \irr(\mathbf{p}')$. In this case, let $\gamma = \beta$.
    \item $\beta = \irr(\mathbf{p}')$. In this case, let $\gamma = \alpha$. 
  \end{itemize}
  Note that $\gamma \notin \{\irr(\mathbf{p}),\irr(\mathbf{p}')\}$ and $\abs{\mathbf{p}''_\gamma - \mathbf{p}'_\gamma} = \abs{\mathbf{p}_\gamma - \mathbf{p}'_\gamma} - \eps$. So, we have
    $$ \abs{\mathbf{p}''_\gamma-\mathbf{p}'_\gamma} \delta_{\gamma\notin \{\irr(\mathbf{p}''),\irr(\mathbf{p}')\}} \leq
    \abs{\mathbf{p}_\gamma-\mathbf{p}'_\gamma} \delta_{\gamma\notin \{\irr(\mathbf{p}),\irr(\mathbf{p}')\}} - \eps.$$
    On the other hand, for all $\sigma\neq\gamma$ in $\Sigma$, we have 
    $$ \abs{\mathbf{p}''_\sigma-\mathbf{p}'_\sigma} \delta_{\sigma\notin \{\irr(\mathbf{p}''),\irr(\mathbf{p}')\}} \leq
    \abs{\mathbf{p}_\sigma-\mathbf{p}'_\sigma} \delta_{\sigma\notin \{\irr(\mathbf{p}),\irr(\mathbf{p}')\}}.$$
    So, by comparing all the terms in the definition of $d_{\mathbf{p},\mathbf{p}'}$ and $d_{\mathbf{p}'',\mathbf{p}'}$, we have $d_{\mathbf{p}'',\mathbf{p}'} \leq d_{\mathbf{p},\mathbf{p}'} -\eps$. So, we are done in this case.
    
    Now, assume that $\irr(\mathbf{p}) = \alpha$. If $\beta\neq \irr(\mathbf{p}')$, let $\gamma = \beta$ and the proof is the same as before. So, we may assume that $\beta = \irr(\mathbf{p}')$. It is not difficult to see that $d_{\mathbf{p}'',\mathbf{p}'} = d_{\mathbf{p},\mathbf{p}'}$. So, we just need to show that $\abs{\Sigma_{\mathbf{p},\mathbf{p}'}} > \abs{\Sigma_{\mathbf{p}'',\mathbf{p}'}}$. Since $\beta=\irr(\mathbf{p}')$, we must have $B = \{\beta\}$, so we have $\mathbf{p}''(\alpha) = \mathbf{p}'(\alpha)$, which means that $\alpha\notin \Sigma_{\mathbf{p}'',\mathbf{p}'}$. It is straightforward to see that $\Sigma_{\mathbf{p}'',\mathbf{p}'} \subseteq \Sigma_{\mathbf{p},\mathbf{p}'}$ and $\alpha\in\Sigma_{\mathbf{p},\mathbf{p}'}$. This completes the proof.
  
\end{proof}

Since $\Gamma_\mathbf{n}$ is connected, without loss of generality we may assume that $\mathbf{p}$ and $\mathbf{p}'$ are adjacent in $\Gamma_\mathbf{n}$.

\subsubsection{Step 2}
\begin{definition}[Sorted trees]
  Let $\Sigma$ be a finite alphabet, and $T$ be a $\Sigma$-labeled $N$-deep tree for some $N\in\N$. We call $T$ {\em sorted} if for every vertex $v\in T$ and every $\sigma\in \Sigma$, if $\sigma$ appears in both $T_v$ and $T_w$, where $w$ is $v$'s sibling, then either $v$ or $w$ is labeled $\sigma$.
\end{definition}

\begin{proposition}
\label{prop:sorted-tree}
  Let $\Sigma$ be a finite alphabet, $\mathbf{n}\in \N^\Sigma$ be a frame on $\Sigma$, $\mathbf{p}\in\PBA_\mathbf{n}$, and $T$ be an $(\mathbf{n}, \mathbf{p})$-uniform tree. There exists a sorted $(\mathbf{n}, \mathbf{p})$-uniform tree $T'$, such that $T\to_\mathbf{n} T'$.
\end{proposition}
\begin{proof}
  Let $\sigma = \irr(\mathbf{p})$, and let $\bar{\sigma}\neq\,\sim$ be a character not in $\Sigma$. To get a sorted $(\mathbf{n},\mathbf{p})$-uniform tree, we just need to make changes to vertices with label $\sigma$. So, if $\sigma =\ \sim$, we are done. Hence, we can assume $\sigma\in\Sigma$. Let $v = c_T(\sigma)$ (recall that $c_T(\sigma)$ is defined in Definition ~\ref{def:n-p-uniform-tree}), and let $w$ be $v$'s sibling. In $T_w$, change all $\sigma$'s to $\bar{\sigma}$ and call the new tree $\bar{T}$.
  
  Let $\bar{\Sigma} = \Sigma \cup \{\bar{\sigma}\}$, $\bar{\Sigma}_{\bar{\mathbf{n}}} = \Sigma$ and $\bar{\mathbf{n}}\in\N^{\bar{\Sigma}_{\bar{\mathbf{n}}}}$ be equal to $\mathbf{n}$, and $\bar{\mathbf{p}}=d_{\bar{T}}$. Let $\bar{z}\in \bar{T}$ be the corresponding vertex to $z\in T$ for any $z\in T$, i.e.\ $I_{\bar{z}}=I_z$. Obviously $\bar{\mathbf{p}}\in \PBA_{\bar{\mathbf{n}}}$, $\irr(\bar{\mathbf{p}})=\ \sim$, and $\bar{T}$ is an $(\bar{\mathbf{n}},\bar{\mathbf{p}})$-uniform tree. Moreover, for a vertex $\bar{u}\in \bar{T}_{\bar{w}}$, if $\bar{\sigma}$ appears in the subtree of $\bar{T}$ rooted at $\bar{u}$'s parent, then $\bar{u}$ is $\bar{\mathbf{n}}$-isolated in $\bar{T}$. So, by Proposition~\ref{prop:expandable-contractible-isolated} we have the following.

  \begin{claim}
  \label{claim:sorted-every-vertex}
    For any $\bar{u}\in \bar{T}_{\bar{w}}$ we have:
    \begin{enumerate}
      \item If $\bar{\sigma}$ appears in $\bar{T}_{\bar{u}}$, then $\bar{u}$ is both $\bar{\mathbf{n}}$-expandable and $\bar{\mathbf{n}}$-super-contractible with respect to $\{\bar{w}\}$.
      \item If $\bar{u}$ is $\bar{\mathbf{n}}$-isolated in $\bar{T}$, then $\bar{u}$ is either $\bar{\mathbf{n}}$-expandable or $\bar{\mathbf{n}}$-contractible with respect to $\{\bar{w}\}$.
    \end{enumerate}
  \end{claim}

  By backward induction on $i=M_\mathbf{n},M_\mathbf{n}-1,\ldots,\depth_w$ we prove the following.
  \begin{claim}
  \label{claim:induction-sorted}
    For $i=M_\mathbf{n},M_\mathbf{n}-1,\ldots,\depth_w$, there is an $(\bar{\mathbf{n}},\bar{\mathbf{p}})$-uniform tree $\bar{T}^{(i)}$ with the following properties.
    \begin{itemize}
        \item $\bar{T}\to_{\bar{\mathbf{n}}} \bar{T}^{(i)}$ with a connecting string $t_{\bar{T},\bar{T}^{(i)}}$ such that the active indices for $(t_{\bar{T},\bar{T}^{(i)}})_{\bar{\mathbf{n}}}$ is a subset of $I_w$.
        \item For $j=i,i+1,\ldots,M_\mathbf{n}$, for at most one of vertices of depth $j$ in $\bar{T}^{(i)}$, say $\bar{u}^{(i)}$, we have that $\bar{\sigma}$ appears in $\bar{T}^{(i)}_{\bar{u}^{(i)}}$ but $\bar{u}^{(i)}$ is not labeled $\bar{\sigma}$ in $\bar{T}^{(i)}$.
    \end{itemize}
  \end{claim}
  \begin{proof}
    For $i=M_\mathbf{n}$, $\bar{T}^{(i)}=\bar{T}$ would work.\\
    For $i<M_\mathbf{n}$, let $\bar{T}^{(i+1)}$ be the resulting tree for $i+1$, which exists by the induction hypothesis, and let $\bar{w}^{(i+1)}\in \bar{T}^{(i+1)}$ be the corresponding vertex to $w\in T$. Enumerate all vertices of depth $i+1$ in $\bar{T}^{(i+1)}$ which have $\bar{\sigma}$ appearing in their subtree by $\bar{u}^{(i+1)}_1,\ldots,\bar{u}^{(i+1)}_{k-1},\bar{u}^{(i+1)}_k$ in a way that $\bar{u}^{(i+1)}_1,\ldots,\bar{u}^{(i+1)}_{k-1}$ are all labeled $\bar{\sigma}$ in $\bar{T}^{(i+1)}$. Note that $\bar{u}^{(i+1)}_1,\ldots,\bar{u}^{(i+1)}_{k-1},\bar{u}^{(i+1)}_k$ are necessarily located in the subtree of $\bar{T}^{(i+1)}$ rooted at $\bar{w}^{(i+1)}$.\\
    Let $\bar{x}^{(i+1)}_1$ be $\bar{u}^{(i+1)}_1$'s sibling. Using Proposition~\ref{prop:2-swap} and Claim~\ref{claim:sorted-every-vertex} we can swap $\bar{x}^{(i+1)}_1$ and $\bar{u}^{(i+1)}_1$ with a connecting string that satisfies the first condition of this claim. So, without loss of generality, we may assume that for $\bar{T}^{(i+1)}$ we can have either one of the following properties:
    \begin{itemize}
        \item $d(\bar{u}^{(i+1)}_2\to \bar{x}^{(i+1)}_1)\leq 1/2\ \diff_{I_{\bar{u}^{(i+1)}_2}}$, or
        \item $d(\bar{x}^{(i+1)}_1\to \bar{u}^{(i+1)}_2)\leq 1/2\ \diff_{I_{\bar{x}^{(i+1)}_1}}$.
    \end{itemize}
    Again, using Proposition~\ref{prop:2-swap}, Claim~\ref{claim:sorted-every-vertex}, and one of the above assumptions, we can swap $\bar{u}^{(i+1)}_2$ and $\bar{x}^{(i+1)}_1$ with a connecting string that satisfies the first condition of this claim. So, again, without loss of generality, we may assume that $\bar{u}^{(i+1)}_1$ and $\bar{u}^{(i+1)}_2$ are siblings in $\bar{T}^{(i+1)}$.\\
    By continuing this process we can assume that $\bar{u}^{(i+1)}_1$ and $\bar{u}^{(i+1)}_2$ are siblings in $\bar{T}^{(i+1)}$, $\bar{u}^{(i+1)}_3$ and $\bar{u}^{(i+1)}_4$ are siblings in $\bar{T}^{(i+1)}$, $\bar{u}^{(i+1)}_5$ and $\bar{u}^{(i+1)}_6$ are siblings in $\bar{T}^{(i+1)}$, and so on. So for each of $\bar{u}^{(i+1)}_1,\ldots,\bar{u}^{(i+1)}_{k-1}$, its parent has label $\bar{\sigma}$. Then $\bar{T}^{(i)} = \bar{T}^{(i+1)}$ works.
  \end{proof}
  
  Let $k=\depth_w$, and let $\bar{T}^{(k)}$ be the $(\bar{\mathbf{n}},\bar{\mathbf{p}})$-uniform tree that is given to us by Claim ~\ref{claim:induction-sorted} for $i=k$. It is not difficult to see that $\bar{T}^{(k)}$ is sorted. We know that $\bar{T}\to_{\bar{\mathbf{n}}} \bar{T}^{(k)}$ with a connecting string $t_{\bar{T},\bar{T}^{(k)}}$ such that the active indices for $(t_{\bar{T},\bar{T}^{(k)}})_{\bar{\mathbf{n}}}$ is a subset of $I_w$.
    
  Let $T'$ be equal to $\bar{T}^{(k)}$ with all labels $\bar{\sigma}$ replaced by $\sigma$. $T'$ is a sorted $(\mathbf{n},\mathbf{p})$-uniform tree.
  Let $t$ be equal to $t_{\bar{T},\bar{T}^{(k)}}$ with all $\bar{\sigma}$'s replaced by $\sigma$. It is easy to see that $\tstring(T)\to_{\mathbf{n}, t}\tstring(T')$, so $T\to_\mathbf{n} T'$.
  This completes the proof.
\end{proof}

So, we can assume that $T$ is sorted.

\subsubsection{Step 3}

So far, we simplified the general case so that we can make the following assumptions: $T$ is a sorted $(\mathbf{n}, \mathbf{p})$-uniform tree, and $\mathbf{p}$ and $\mathbf{p}'$ are adjacent in $\Gamma_\mathbf{n}$. Since $\mathbf{p},\mathbf{p}'$ are adjacent in $\Gamma_\mathbf{n}$, there are characters $\alpha, \beta\in\Sigma$ such that $\mathbf{p}_\alpha > \mathbf{p}'_\alpha$, $\mathbf{p}_\beta < \mathbf{p}'_\beta$, $\mathbf{p}_\sigma = \mathbf{p}'_\sigma$ for $\sigma\in\Sigma\setminus\{\alpha,\beta\}$, and $\irr(\mathbf{p}),\irr(\mathbf{p}')\in \{ \sim,\alpha,\beta \}$.

In this step, we introduce two general scenarios, and by using Proposition~\ref{prop:2-swap} for scenario (I) and Proposition~\ref{prop:4-swap} for scenario (II), we will show how to get an $(\mathbf{n}, \mathbf{p}')$-uniform tree $T'$ with $T\to_\mathbf{n} T'$. In the next step, we will show how to reduce the general case to one of these scenarios in different cases.

Now we introduce the two scenarios. Let $v_\alpha=c_T(\alpha)$ and $w_\alpha$ be $v_\alpha$'s sibling. Let $v_\beta=c_T(\beta)$ and $w_\beta$ be $v_\beta$'s sibling. Recall the definition of $c_T(\sigma)$ from Definition ~\ref{def:n-p-uniform-tree}. Let $\bar{\alpha}, \bar{\beta}\neq\;\sim$ be two characters not in $\Sigma$.

\vspace{0.3cm}
{\em Scenario (I):
  Let $\bar{\Sigma} = \Sigma \cup \{\bar{\alpha}, \bar{\beta}\}$, $\bar{\Sigma}_{\bar{\mathbf{n}}} = \Sigma$ and $\bar{\mathbf{n}}\in \N^{\bar{\Sigma}_{\bar{\mathbf{n}}}}$ be equal to $\mathbf{n}$.
  Assume that we have $x\in T_{w_\alpha}$ and $y\in T_{w_\beta}$ with the same depth. Let $\bar{T}$ be equal to $T$ everywhere except that all $\alpha$'s in $\bar{T}_{\bar{x}}$ are changed to $\bar{\alpha}$, and all the $\beta$'s in $\bar{T}_{\bar{y}}$ are changed to $\bar{\beta}$, where for each $v\in T$ we denote its corresponding vertex in $\bar{T}$ by $\bar{v}$.\\
  Moreover, assume that $\bar{y}$ is $\bar{\mathbf{n}}$-isolated in $\bar{T}$, and for $\bar{\mathbf{p}}=d_{\bar{T}}$ we have $\bar{\mathbf{p}}_{\bar{\alpha}} - \bar{\mathbf{p}}_{\bar{\beta}} = \mathbf{p}_\alpha-\mathbf{p}'_\alpha = \mathbf{p}'_\beta-\mathbf{p}_\beta.$}
\begin{enumerate}
  \item Let $z$ be the parent of $x$.
  
  \item Note that in this scenario, in addition to $\bar{y}$, which we assume is $\bar{\mathbf{n}}$-isolated in $\bar{T}$, $\bar{x}$ is also $\bar{\mathbf{n}}$-isolated in $\bar{T}$. Moreover, $\bar{T}$ is $(\bar{\mathbf{n}},\bar{\mathbf{p}})$-uniform.
  
  \item Since (i) we have at least one occurrence of $\bar{\alpha}$ in $\bar{T}_{\bar{x}}$, (ii) $\bar{x}$ is $\bar{\mathbf{n}}$-isolated in $\bar{T}$, and (iii) $\bar{T}$ is $(\bar{\mathbf{n}},\bar{\mathbf{p}})$-uniform and $\irr(d_{\bar{T}})$ does not appear in $\bar{T}_{\bar{x}}$, by the third part of Proposition~\ref{prop:expandable-contractible-isolated} we get that $\bar{T}_{\bar{x}}$ is both $\bar{\mathbf{n}}$-expandable and $\bar{\mathbf{n}}$-super-contractible with respect to $\{\troot_{\bar{T}}\}$, and so, also with respect to $\{\bar{x},\bar{y}\}$.\\
  Similarly, $\bar{y}$ is either $\bar{\mathbf{n}}$-contractible or $\bar{\mathbf{n}}$-expandable with respect to $\{\bar{x},\bar{y}\}$.
 
  \item Similar to the proof of Proposition~\ref{prop:sorted-tree}, we can show that we can swap $\bar{x}$ and its sibling. So, without loss of generality we can assume either one of the following assumptions (but only one of them): (i) $d(\bar{x}\to \bar{y})\leq 1/2\ \diff_{I_x}$, or (ii) $d(\bar{y}\to \bar{x})\leq 1/2\ \diff_{I_y}$.\\
  More precisely, there is a tree $\bar{T}^{(*)}$ with $\bar{T}\to_{\bar{\mathbf{n}}} \bar{T}^{(*)}$, which is isomorphic to $\bar{T}$, and in $\bar{T}^{(*)}$ we have $d(\bar{x}^{(*)}\to \bar{y}^{(*)})\leq 1/2\ \diff_{I_x}$, where $\bar{x}^{(*)}$ and $\bar{y}^{(*)}$ are the images of $\bar{x}$ and $\bar{y}$ under the isomorphism of $\bar{T}$ and $\bar{T}^{(*)}$. Similarly, there is a tree $\bar{T}^{(+)}$ with $\bar{T}\to_{\bar{\mathbf{n}}} \bar{T}^{(+)}$, which is isomorphic to $\bar{T}$, and in $\bar{T}^{(+)}$ we have $d(\bar{y}^{(+)}\to \bar{x}^{(+)})\leq 1/2\ \diff_{I_y}$. Moreover, we can get each of these trees with a connecting string whose active indices is a subset of $I_z \mod 2^{M_\mathbf{n}}$.
 
  \item So, by Proposition~\ref{prop:2-swap}, and the previous two results, we can swap $\bar{x}$ and $\bar{y}$ in $\bar{T}$. Let $\bar{T}^{(0)}$ be the resulting tree. So we have $\bar{T}\to_{\bar{\mathbf{n}}} \bar{T}^{(0)}$ with a connecting string whose active indices is a subset of $I_z\cup I_y \mod 2^{M_\mathbf{n}}$. Let $t_{\bar{T},\bar{T}^{(0)}}$ be one such connecting string for $(\bar{T},\bar{T}^{(0)})_{\bar{\mathbf{n}}}$. So, if $I_{\bar{T},\bar{T}^{(0)}}$ is the set of active indices for $(t_{\bar{T},\bar{T}^{(0)}})_\mathbf{n}$ we have $I_{\bar{T},\bar{T}^{(0)}}\subseteq I_z\cup I_y \subseteq I_{w_\alpha} \cup I_{v_\alpha} \cup I_{w_\beta} \mod 2^{M_\mathbf{n}}$.
  
  Obviously $\bar{T}^{(0)}$ is an $(\bar{\mathbf{n}},\bar{\mathbf{p}})$-uniform tree. So, in particular, $d_{\bar{T}^{(0)}}=d_{\bar{T}}$.
 
  \item In $\bar{T}^{(0)}$ change all the $\bar{\alpha}$'s to $\beta$, and change all the $\bar{\beta}$'s to $\alpha$. Call the new tree $T'$. Note that for $\sigma\in\Sigma\setminus\{\alpha,\beta\}$ we have $d_{T'}(\sigma) = d_T(\sigma) = \mathbf{p}'_\sigma$, and
  \begin{align*}
    d_{T'}(\alpha) &= \frac{\text{num}_{T'}(\alpha)}{2^{M_\mathbf{n}}} =\frac{\text{num}_{\bar{T}^{(0)}}(\alpha)+\text{num}_{\bar{T}^{(0)}}(\bar{\beta})}{2^{M_\mathbf{n}}}\\
    &= \frac{\text{num}_T(\alpha)-\text{num}_{\bar{T}}(\bar{\alpha})+\text{num}_{\bar{T}}(\bar{\beta})}{2^{M_\mathbf{n}}}\\
    &= \mathbf{p}_\alpha - (\mathbf{p}_\alpha-\mathbf{p}'_\alpha) = \mathbf{p}'_\alpha,\\
    d_{T'}(\beta) &= \frac{\text{num}_{T'}(\beta)}{2^{M_\mathbf{n}}} =\frac{\text{num}_{\bar{T}^{(0)}}(\beta)+\text{num}_{\bar{T}^{(0)}}(\bar{\alpha})}{2^{M_\mathbf{n}}}\\
    &= \frac{\text{num}_T(\beta)-\text{num}_{\bar{T}}(\bar{\beta})+\text{num}_{\bar{T}}(\bar{\alpha})}{2^{M_\mathbf{n}}}\\
    &= \mathbf{p}_\beta - (\mathbf{p}_\beta-\mathbf{p}'_\beta) = \mathbf{p}'_\beta.
  \end{align*}
  So $d_{T'}=\mathbf{p}'$. It is not difficult to see that $T'$ is $(\mathbf{n},\mathbf{p}')$-uniform.
  
  \item We know that $t_{\bar{T},\bar{T}^{(0)}}\in\Sigma^{[k2^{M_\mathbf{n}}]}$ for some $k\in\N$. 
  Define $t_{T,T'}\in \Sigma^{[k2^{M_\mathbf{n}}]}$ as follows:
  $$ t_{T,T'}(i) =
  \begin{cases}
    t_{\bar{T},\bar{T}^{(0)}}(i)& \text{if } t_{\bar{T},\bar{T}^{(0)}}(i)\notin \{\bar{\alpha},\bar{\beta}\},\\
    \alpha& \text{if } t_{\bar{T},\bar{T}^{(0)}}(i)\in\{\bar{\alpha},\bar{\beta}\} \text{ and } i \in I_{w_\alpha}\cup I_{v_\alpha}\setminus I_y \mod 2^{M_\mathbf{n}},\\
    \beta& \text{if } t_{\bar{T},\bar{T}^{(0)}}(i)\in\{\bar{\alpha},\bar{\beta}\} \text{ and } i \in I_y \mod 2^{M_\mathbf{n}}.
  \end{cases}
  $$
  Since $I_{\bar{T},\bar{T}^{(0)}}\subseteq I_{w_\alpha} \cup I_{v_\alpha} \cup I_{w_\beta} \mod 2^{M_\mathbf{n}}$, $t_{T,T'}$ is well defined. It is not difficult to see that 
  \begin{itemize}
    \item $\Sigma_{t_{T,T'}} = \Sigma_{\tstring(T)} = \Sigma_{\tstring(T')}$,
    \item $t_{T,T'}$ is $\mathbf{n}$-compatible,
    \item $t_{T,T'}[\mathbf{n}, +1] = \tstring(T)$, and
    \item $t_{T,T'}[\mathbf{n}, -1] = \tstring(T')$.
  \end{itemize}
  It means that $T\to_{\mathbf{n},t_{T,T'}} T'$. Hence $T\to_\mathbf{n} T'$.
\end{enumerate}

\vspace{0.3cm}
{\em Scenario (II):
Assume that 
\begin{itemize}
  \item $x_1\in T\setminus T_{w_\beta}$ is an ancestor of $u_\alpha$, where $u_\alpha$ is the parent of $w_\alpha$ and $v_\alpha$,
  \item $y_2\in T_{w_\beta}$ and $y_1$ is the sibling of $y_2$.
  \item $x_1$, $y_1$, and $y_2$ have the same depth,
  \item $y_2$ is $\mathbf{n}$-isolated in $T$, and 
  \item $\beta$ appears in $T_{y_1}$.
\end{itemize}
}

\begin{enumerate}

  \item Let $x_2$ be $x_1$'s sibling, $x$ be the parent of $x_1$ and $x_2$, and $y$ be the parent of $y_1$ and $y_2$.
  
  \item Let $\bar{\Sigma}= \Sigma \cup \{\bar{\beta}\}$, $\bar{\Sigma}_{\bar{\mathbf{n}}} = \Sigma$ and $\bar{\mathbf{n}}\in\N^{\bar{\Sigma}_{\bar{\mathbf{n}}}}$ be equal to $\mathbf{n}$. Let $\bar{T}$ be equal to $T$ with the following exception: if $y_1\neq v_\beta$ change all the $\beta$'s in $\bar{T}_{\bar{y_1}}$ to $\bar{\beta}$, where $\bar{z}\in\bar{T}$ is the corresponding vertex to $z\in T$ for any $z\in T$. Let $\bar{\mathbf{p}}=d_{\bar{T}}$. Obviously $\bar{T}$ is an $(\bar{\mathbf{n}},\bar{\mathbf{p}})$-uniform tree.
  
  \item Since $y_2\in T_{w_\beta}$ is $\mathbf{n}$-isolated in $T$, the only character that could appear both in $T_{y_1}$ and outside of $T_{y_1}$ is $\beta$. So $\bar{y_1}$ is $\bar{\mathbf{n}}$-isolated in $\bar{T}$. Also, since $y_2$ is $\mathbf{n}$-isolated in $T$, we get that $\bar{y_2}$ is $\bar{\mathbf{n}}$-isolated in $\bar{T}$.\\
  Since $\bar{x_1}, \bar{x_2} \notin \bar{T}_{\bar{w_\beta}}$, we get that $\bar{x_1}$ and $\bar{x_2}$ are $\bar{\mathbf{n}}$-isolated in $\bar{T}$.
  
  \item By the second part of Proposition~\ref{prop:expandable-contractible-isolated}, we can see that $\bar{x_1}$ is $\bar{\mathbf{n}}$-expandable with respect to $\{\troot_{\bar{T}}\}$, and so, with respect to $\{\bar{x},\bar{y}\}$.\\
  Similarly, by the third part of Proposition~\ref{prop:expandable-contractible-isolated}, $\bar{y_1}$ is $\bar{\mathbf{n}}$-contractible with respect to $\{\bar{y}\}$.
  
  \item The previous results show that the conditions for Proposition~\ref{prop:4-swap} hold here. So, there is a permutation $\pi$ of $\bar{V}=\{\bar{x_1},\bar{x_2},\bar{y_1},\bar{y_2}\}$ with $\pi(\bar{x_1}),\pi(\bar{y_1})\in\{\bar{y_1},\bar{y_2}\}$, and an $\mathbf{n}$-compatible tree $\bar{T}'$ such that
  \begin{itemize}
    \item $\bar{T}'_{\bar{v}'}\cong \bar{T}_{\pi^{-1}(\bar{v})}$ for $\bar{v}\in \bar{V}$, where $\bar{w}'\in \bar{T}'$ is the corresponding vertex to $w$ for any $w\in T$, and
    \item $\bar{T}\to_\mathbf{n} \bar{T}'$ and for some connecting string $t_{\bar{T},\bar{T}'}$ for $(\bar{T},\bar{T}')_{\bar{\mathbf{n}}}$ we have $\IND_{\bar{\mathbf{n}}}(t_{\bar{T},\bar{T}'}) \subseteq \cup_{\bar{v}\in \bar{V}} I_{\bar{v}} = \cup_{v\in V}I_v \mod 2^{M_\mathbf{n}}$ and $\IND_{\bar{\mathbf{n}}}(t_{\bar{T},\bar{T}'}, \bar{y_1}) \subseteq I_{\bar{y}} = I_y \mod 2^{M_\mathbf{n}}$. In particular, for 
    $$J_{\bar{\beta}} = \{l\in\domain(t_{\bar{T},\bar{T}'})\;|\; t_{\bar{T},\bar{T}'}(l) = \bar{\beta}\},$$
    we have $J_{\bar{\beta}}\subseteq I_y \mod 2^{M_\mathbf{n}}$.
  \end{itemize}
  
  \item Let $T'$ be equal to $\bar{T}'$ with all the $\bar{\beta}$'s replaced by $\beta$, and $t_{T,T'}$ be equal to $t_{\bar{T},\bar{T}'}$ with all the $\bar{\beta}$'s replaced by $\beta$. It is straightforward to check that $T'$ is an $(\mathbf{n},\mathbf{p})$-uniform tree and 
  \begin{itemize}
    \item $\Sigma_{t_{T,T'}} = \Sigma_{\tstring(T)} = \Sigma_{\tstring(T')}$,
    \item $t_{T,T'}$ is $\mathbf{n}$-compatible,
    \item $t_{T,T'}[\mathbf{n}, +1] = \tstring(T)$, and
    \item $t_{T,T'}[\mathbf{n}, -1] = \tstring(T')$.
  \end{itemize}
  It means that $T\to_{\mathbf{n}, t_{T,T'}} T'$. So, $T\to_\mathbf{n} T'$.
  
  \item In $T'$, let $v_\alpha'=c_{T'}(\alpha)$, $v_\beta'=c_{T'}(\beta)$, $w_\alpha'$ be $v_\alpha'$'s sibling, and $w_\beta'$ be $v_\beta'$'s sibling. We can see that $v_\alpha',w_\alpha'\in T'_{w_\beta'}$.\\
  Let $k = (\mathbf{p}_\alpha-\mathbf{p}'_\alpha)2^{M_\mathbf{n}}$. Obviously we have at least $k$ appearances of $\alpha$ in $T_{w_\alpha'}$. Choose $k$ leaves in $T_{w_\alpha'}$ with label $\alpha$ and change their labels to $\beta$. Call the new tree $T''$. Obviously $T'\to_\mathbf{n} T''$ (because $\tstring(T')\wedge_\mathbf{n}\tstring(T'')$ works as a connecting string), so we have $T\to_\mathbf{n} T''$. It is straightforward to see that $T''$ is an $(\mathbf{n},\mathbf{p}')$-uniform tree.
\end{enumerate}

\subsubsection{Step 4}

In this step, based on the values of $\irr(\mathbf{p})$ and $\irr(\mathbf{p}')$ and the structure of $T$, we will consider different cases and see that each case falls into one of the two scenarios we introduced in the previous step. That will conclude the proof of Lemma~\ref{lem:n-p-good-strings}.

Consider the following cases:

\vspace{0.3cm}
{\bf Case 1: $\irr(\mathbf{p}')\neq \alpha$.}

Let $k = \text{num}_{T_{w_\alpha}}(\alpha)$. In $T$ we have to change the labels of $k$ leaves with label $\alpha$ to $\beta$. So, in $T_{w_\beta}$ we have at least $k$ leaves with labels different from $\beta$.

Let $\Sigma_\mathbf{m} = \Sigma\setminus \{\beta\}$ and $\mathbf{m}$ be equal to the restriction of $\mathbf{n}$ to $\Sigma_\mathbf{m}$. We say that a vertex $v\in T_{w_\beta}$ is {\em interesting} if
\begin{itemize}
  \item $\beta$ appears in the sibling of $v$,
  \item $v$ is $\mathbf{m}$-isolated, and
  \item $\text{num}_{T_v}(\beta)\leq \abs{I_v} - k$.
\end{itemize}
If $w_\beta$ is not interesting, then $\alpha$ appears in $T_{w_\beta}$. Since $T$ is sorted, we get that all the $\alpha$'s in $T_{w_\alpha}$ are in $T_{w_\beta}$. In this case, change all the $\alpha$'s in $T_{w_\alpha}$ to $\beta$ and call the new tree $T'$. It is easy to see that $T'$ satisfies the conditions we want.\\
So, assume that $w_\beta$ is interesting. So, there exists a deepest interesting vertex $b$ in $T_{w_\beta}$. We claim that $\beta$ does not appear in $T_b$. If $\beta$ appears in $T_b$, then $\irr(\mathbf{p}) = \beta$, which means that $w_\alpha$ is labeled $\alpha$ in $T$, therefore $k$ is a power of two. Since $T$ is sorted, if $k$ is a power of two, $T_b$ does not have any $\beta$'s. This is a contradiction. So, $\beta$ does not appear in $T_b$. This means that $b$ is $\mathbf{n}$-isolated.

Since $T$ is sorted, there is a vertex $a\in T$ with the same depth as $b$ such that all the $\alpha$'s in $T_{w_\alpha}$ are also in $T_a$. If $a\in T_{w_\beta}$, change all the $\alpha$'s in $T_a$ to $\beta$ and call the new tree $T'$. It is easy to see that $T'$ satisfies the conditions we want. So, we can assume $a\in T\setminus T_{w_\beta}$. Consider the following two cases.
\begin{itemize}
  \item $\depth_a < \depth_{w_\alpha}$: In this case we fall into scenario (II). Obviously $a\in T\setminus T_{w_\beta}$ is an ancestor of the parent of $v_\alpha$ and $w_\alpha$, and $b\in T_{w_\beta}$, and $a,b$ have the same depth. Moreover, $b$ is $\mathbf{n}$-isolated in $T$ and since $b$ is interesting, $\beta$ appears in the sibling of $b$. So, $x_1=a, y_2=b$ works for scenario (II).
  \item $\depth_a \geq \depth_{w_\alpha}$: So $a\in T_{w_\alpha}$ and $b\in T_{w_\beta}$. $x=a, y=b$ works for scenario (I).
\end{itemize}

\vspace{0.3cm}
{\bf Case 2:} $\irr(\mathbf{p}') = \alpha$.\\
Let $k = \abs{I_{w_\beta}} - \text{num}_{T_{w_\beta}}(\beta)$. In $T$ we have to change the labels of $k$ leaves with label $\alpha$ to $\beta$. So, in $T_{w_\beta}$ we exactly $k$ leaves with labels different from $\beta$. Let $b$ be the deepest vertex in $T_{w_\beta}$ with $\text{num}_{T_b}(\beta)= \abs{I_b} - k$. Similarly, in $T_{w_\alpha}$ we have at least $k$ leaves with label $\alpha$. Let $a$ be the deepest vertex in $T_{w_\alpha}$ with $\text{num}_{T_a}(\alpha) \geq k$. Since $T$ is sorted, it is not difficult to see that $a$ and $b$ have the same depth.
\begin{itemize}
  \item If $\beta$ appears in $T_b$, it means that $\irr(\mathbf{p})=\beta$, which implies that $w_\alpha$ is labeled with $\alpha$, and hence $a$ is labeled $\alpha$. So $\text{num}_{T_a}(\alpha) = \abs{I_a} = \abs{I_b}$. We also know that $\text{num}_{T_b}(\beta) = \abs{I_b} - k$. So $\text{num}_{T_a}(\alpha) - \text{num}_{T_b}(\beta) = k$.
  \item If $\beta$ does not appear in $T_b$, it means that $k=\abs{I_b}$. So $a$ is labeled $\alpha$ in $T$. So $\text{num}_{T_a}(\alpha) - \text{num}_{T_b}(\beta) = \abs{I_b} - 0 = k$
\end{itemize}
It is straightforward to see that $x=a, y=b$ works for scenario (I).

\vspace{1cm}
\section{Proof of Theorem ~\ref{thm:eps-quasi-regular}}

In this section we prove Theorem~\ref{thm:eps-quasi-regular}. The proof uses a series of reductions.


\begin{definition}[Colorings]
A {\em coloring} of $\N$ is a family $\{f_c\}_{c \in C}$ of strictly increasing functions from $\N$ to $\N$ such that for every $m \in \N$, there are unique $c \in C$ and $i \in \N$ for which $f_c(i) = m$. We refer to the set $C$ as the set of colors.
\end{definition}

A coloring $\{f_c\}_{c\in C}$ of $\N$ corresponds to a sequence $s\in C^\N$: for $m\in \N$, $s(m) = c$ where $c$ is the unique color such that $m\in f_c(\N)$. Using this correspondence, we can define {\em well-distributed} colorings: a coloring $\{f_c\}_{c\in C}$ is well-distributed if the corresponding sequence $s\in C^\N$ is well-distributed in the sense of Definition~\ref{def:strings}. If the coloring $\{f_c\}_{c\in C}$ is well-distributed, define its density to be equal to the density function of the corresponding sequence.

A strictly increasing function $f\colon\N\to\N$ is called {\em well-distributed} if the following limit exists 
$$d(f) \coloneqq \lim_{n \rightarrow \infty}{\frac{\abs{f^{-1}\big([n]\big)}}{n}}.$$
Note that a coloring $\{f_c\}_{c\in C}$ is well-distributed if and only if each $f_c$ is well-distributed. In this case, $d(f_c)$ is equal to $d_s(c)$, where $d_s$ is the density function of the corresponding sequence $s$.

\begin{definition}[$k$-quasi-regularity]
\label{def:k-q-r}
For $k\in\N$, the {\em $k$-quasi-regularity} of a strictly increasing function $f \colon \N \to \N$, denoted by $\text{QR}_k(f)$, is the supremum across $q,r \geq k$ and $n,m \in \N$ of
$$\frac{q}{r} \frac{f(n+r) - f(n)}{f(m+q) - f(m)}.$$
\end{definition}

For a strictly increasing function $f:\N\to\N$, we define the quasi-regularity of $f$ as
$\text{QR}(f) \coloneqq \sup_{n,m\in\N} \frac{f(n+1)-f(n)}{f(m+1)-f(m)}$.
It is not difficult to prove the following claim.

\begin{claim}
\label{claim:min-gap}
  Let $f\colon \N\to\N$ be strictly increasing. The following hold.
  \begin{enumerate}
    \item $\text{QR}(f) = \text{QR}_1(f)$.
    \item If, in addition, $f$ is well-distributed, then
    $$\min_{n\in\N} f(n+1)-f(n) \geq \frac{1}{d(f) \text{QR}(f)}.$$
  \end{enumerate}
\end{claim}

Let $\{f_c\}_{c\in C}$ be a coloring of $\N$ and $s\in C^\N$ be the corresponding sequence for this coloring. It is straightforward to see that $\text{QR}(s) = \sup_{c\in C} \text{QR}(f_c) = \sup_{c\in C} \text{QR}_1(f_c)$, where $\text{QR}(s)$ is defined in Definition~\ref{def:quasi-reg seq}.

The following lemma tells us how quasi-regularity behaves under composition. This lemma helps us prove Theorem~\ref{thm:eps-quasi-regular} by a series of reductions.

\begin{lemma}
\label{lemma:composition}
Let $f,g\colon \N\to\N$ be strictly increasing. Suppose that $g(n+1)-g(n)\geq k/l$ for all $n\in\N$, with $k,l\in \N$. Then
$$\text{QR}_l(f \circ g) \leq \text{QR}_k(f) \cdot \text{QR}_l(g).$$
\end{lemma}

\begin{proof}
Fix $n,m\in\N$ and $r,q\in\N$ with $r,q\geq l$, we need to show that: 
$$\frac{q}{r}\frac{f(g(n+r))-f(g(n))}{f(g(m+q))-f(g(m))} \leq \text{QR}_k(f) \cdot \text{QR}_l(g).$$
Since $g(n+1)-g(n)\geq k/l$ for all $n\in\N$, it follows that $g(n+r)-g(n) \geq r\frac{k}{l} \geq k$ and $g(m+q)-g(m) \geq q\frac{k}{l} \geq k$.
So, $g(n+r) = g(n) + r'$ and $g(m+q) = g(m) + q'$ with $r', q'\geq k$.
\begin{align*}
  \frac{q}{r} \frac{f(g(n+r))-f(g(n))}{f(g(m+q))-f(g(m))} & = \frac{q}{r} \frac{f(g(n)+r')-f(g(n))}{f(g(m)+q')-f(g(m))} \\
  & = \left(\frac{q}{r} \frac{r'}{q'} \right) \left( \frac{q'}{r'} \frac{f(g(n)+r')-f(g(n))}{f(g(m)+q')-f(g(m))} \right) \\
  & = \left(\frac{q}{r} \frac{g(n+r)-g(n)}{g(m+q)-g(m)}\right) \left( \frac{q'}{r'} \frac{f(g(n)+r')-f(g(n))}{f(g(m)+q')-f(g(m))} \right)\\
  & \leq \text{QR}_l(g) \cdot \text{QR}_k(f)
\end{align*}
\end{proof}

The following lemma is an immediate corollary of the main result in ~\cite{tijdeman1973distribution}, which we will use later in some of the proofs.

\begin{lemma}
\label{lem:one-sided-to-two-sided}
Suppose $\Sigma$ is a countable alphabet and $\mathbf{p}$ is a probability distribution on $\Sigma$. Then there is an infinite sequence $s\in\Sigma^\N$ with:
\begin{enumerate}
    \item $d_s = \mathbf{p}$ and
    \item for all $M,N\in\N$, and $\sigma\in\Sigma$
    \begin{align*}
        \abs{ \frac{ \abs{s^{-1}(\{\sigma \}) \cap [M, M+N)} }{N} - \mathbf{p}_\sigma } < \frac{2}{N}
    \end{align*}
\end{enumerate}
\end{lemma}

\begin{proof}
  By ~\cite{tijdeman1973distribution}*{Theorem 1}, there exists a sequence $s\in\Sigma^\N$ with $d_s = \mathbf{p}$ and 
  \begin{align}
  \label{eq:ref_thm}
    \abs{ \abs{s^{-1}(\{\sigma \}) \cap [K]} - K \mathbf{p}_\sigma } < 1
  \end{align}
  for all $K\in\N$ and $\sigma\in\Sigma$.
  
  To complete the proof, we need to show 
  $$\abs{ \frac{ \abs{s^{-1}(\{\sigma \}) \cap [M, M+N)} }{N} - \mathbf{p}_\sigma } < \frac{2}{N}$$
  for all $N\in\N$, $M\in\N$, and $\sigma\in\Sigma$.
  
  If $M = 1$, the result follows from putting $K = N$ and dividing both sides of ~\eqref{eq:ref_thm} by $N$. So, let $M\geq 2$. Inequality ~\eqref{eq:ref_thm} for $K = M-1$ and $K = N+M-1$ gives us
  \begin{align*}
    & \abs{ \abs{s^{-1}(\{\sigma \}) \cap [M-1]} - (M-1) \mathbf{p}_\sigma } < 1,\\
    & \abs{ \abs{s^{-1}(\{\sigma \}) \cap [N+M-1]} - (N+M-1) \mathbf{p}_\sigma } < 1.
  \end{align*}
  From these two bounds, we get
  $$\abs{ \abs{s^{-1}(\{\sigma \}) \cap [M, M+N)} - N \mathbf{p}_\sigma } < 2,$$
  which completes the proof.
\end{proof}

Now, we proceed to consider some special cases, and show that how, in each case, we can get a partial result, i.e.\ a coloring with small quasi-regularity for most of the colors. Later, we will combine these partial results and prove Theorem ~\ref{thm:eps-quasi-regular}.

\subsection{Case I} Here we assume that the probabilities in the distribution are almost equal.

\begin{proposition}
\label{prop:big_1}
Suppose $n>2^9$ and $(p_1, \ldots, p_n)$ is a probability distribution with $p_1 \geq p_2 \geq \dots \geq p_n$, $p_1 / p_n \leq 2$. Then, there exists a coloring $\{ f_i \}_{i \in \{ 1,\ldots,n \} }$ of $\N$ such that for $i\in\{ 1,\ldots,n \}$
\begin{enumerate}
    \item $d(f_i) = p_i$, and
    \item $\text{QR}_1(f_i) < 1 + 64 \sqrt{\frac{\log n}{n}}$.
\end{enumerate}
\end{proposition}

This proposition and its proof are very similar to ~\cite{kempe2018quasi}*{Theorem 6.1}.

\begin{proof}
Let $n>2^9$ and $\mathbf{p}=(p_1,\ldots,p_n)$. For $i=1,2,\ldots,n$, let $k_i\in\N\cup\{0\}$ be such that $p_i\in[\frac{k_i}{n^4},\frac{k_i+1}{n^4}]$. Since $p_1/p_n\leq 2$, we get that $p_i\in[\frac{1}{2n}, \frac{2}{n}]$ and $k_i\in[\floor{\frac{n^3}{2}}, 2n^3]$. By ~\cite{kempe2018quasi}*{Lemma 4.3}, there are probability distributions $\mathbf{q}^1,\ldots,\mathbf{q}^r$ and $0\leq\alpha_1,\ldots,\alpha_r\leq 1$ with:
\begin{enumerate}
    \item $\sum_{j=1}^{r} \alpha_j = 1$.
    \item $\sum_{j=1}^{r} \alpha_j \mathbf{q}^j = \mathbf{p}$.
    \item If $\mathbf{q}^j = (q^j_1,\ldots,q^j_n)$, we have $q^j_i\in[\frac{k_i}{n^4},\frac{k_i+1}{n^4}]$ for all $i=1,\ldots,n$.
    \item If $\mathbf{q}^j = (q^j_1,\ldots,q^j_n)$, we have $q^j_i\in\{\frac{k_i}{n^4},\frac{k_i+1}{n^4}\}$ for all but at most one $i=1,\ldots,n$.
\end{enumerate}
Item (4) implies that $q^j_i=\frac{l^j_i}{n^4}$ for some $l^j_i\in \Z$ for all $i=1,\ldots,n$, and item (3) shows that $l^j_i\in[k_i, k_i+1]$ for all $i=1,\ldots,n$. So, $\mathbf{p}$ is a convex combination of $\mathbf{q}^1,\ldots,\mathbf{q}^r$, where for each $\mathbf{q}^j=(q^j_1,\ldots,q^j_n)$ we have $q^j_i\in\{\frac{k_i}{n^4},\frac{k_i+1}{n^4}\}$ for all $i=1,\ldots,n$. Since all $\mathbf{q}^j$ are in a $(n-1)$-dimensional subspace of $\R^n$, we can further assume that $\mathbf{p}$ is in the convex hull of at most $n$ such points $\mathbf{q}^j$. So, without loss of generality, we can assume $r\leq n$.

For $j=1,\ldots,r$ and $i=1,\ldots,n$ let $\theta^j_i$ be independent uniform random variables with $\theta^j_i\in[0,q^j_i]$. For each $j=1,\ldots,r$, we build a random bipartite graph $G^j_{\theta^j_1,\ldots,\theta^j_n}$ as a function of $\theta^j_1,\ldots,\theta^j_n$ and for $\delta = \sqrt{\frac{2\log n}{n^7}}$ show
\begin{enumerate}
    \item Each $G^j_{\theta^j_1,\ldots,\theta^j_n}$ has a perfect matching with high probability,
    \item There are common values $\eta_1\in[\delta,\frac{k_1}{n^4}-\delta],\ldots,\eta_n\in[\delta,\frac{k_n}{n^4}-\delta]$ such that $G^j_{\eta_1,\ldots,\eta_n}$ has a perfect matching for all $j=1,\ldots,r$, and
    \item When for $\eta_1\in[\delta,\frac{k_1}{n^4}-\delta],\ldots,\eta_n\in[\delta,\frac{k_n}{n^4}-\delta]$ every $G^j_{\eta_1,\ldots,\eta_n}$ has a perfect matching, we would get a coloring $\{f_i\}_{i\in\{1,\ldots,n\}}$ of $\N$ with $d(f_i)=p_i$ such that $\text{QR}_1(f_i)$ is close to 1 for each $i=1,\ldots,n$.
\end{enumerate}

We start by defining the bipartite graphs $G^j_{\theta^j_1,\ldots,\theta^j_n}$. For simplicity, we denote the graph by $G^j$. Let $\delta = \sqrt{\frac{2\log n}{n^7}}$ and
\begin{align*}
    V &= \{0, 1,\ldots,n^4-1\}  \text{, and}\\
    W^j &= \{(i,\ell)\ | \ i=1,\ldots,n \text{ and } 0\leq \ell < q^j_i n^4\}.
\end{align*}
Define $G^j$ on the set of vertices $V\cup W^j$ by adding an edge between $t\in V$ and $(i,\ell)\in W^j$ if
\begin{align}
\label{eq:graph definition}
    d(\frac{t}{n^4},\, \theta^j_i + \frac{\ell}{q^j_i n^4}) \leq \delta,
\end{align} where $\theta^j_i + \frac{\ell}{q^j_i n^4}$ and $d(\cdot,\cdot)$ are calculated mod 1.

\begin{claim}
\label{claim:prob matching}
  Let $E^j$ be the event that $G^j_{\theta^j_1,\ldots,\theta^j_n}$ has no perfect matchings. Then probability of $E^j$ is at most $1/n^8$.
\end{claim}
\begin{proof}
  This claim follows from the proof of ~\cite{kempe2018quasi}*{Lemma 6.4} when $M=n^4$.
\end{proof}

\begin{claim}
\label{claim:common matching}
There are common values $\eta_1\in[\delta,\frac{k_1}{n^4}-\delta],\ldots,\eta_n\in[\delta,\frac{k_n}{n^4}-\delta]$ such that $G^j_{\eta_1,\ldots,\eta_n}$ has a perfect matching for all $j=1,\ldots,r$.
\end{claim}
\begin{proof}
To show the existence of such $\eta_1,\ldots,\eta_n$, by the union bound, we just need to show that 
\begin{align*}
    \sum_{j=1}^{r} \Pr{ \neg\left(\theta^j_1\in [\delta,\frac{k_1}{n^4}-\delta], \ldots, \theta^j_n\in [\delta,\frac{k_n}{n^4}-\delta]\right) \text{ or } E^j} < 1
\end{align*}
Note that 
\begin{align*}
    \Pr{\theta^j_i\in[\delta,\frac{k_i}{n^4}-\delta]}
    & = (\frac{k_i}{n^4} - 2\delta)/q^j_i \\
    & \geq (\frac{k_i}{n^4} - 2\delta)/ (\frac{k_i+1}{n^4}) = \frac{k_i - 2\delta n^4}{k_i+1} \\
    & \geq \frac{\floor{n^3/2} - \sqrt{8n\log n}}{\floor{n^3/2}+1} \\
    & \geq 1-\frac{4}{n^{9/4}},
\end{align*} so
\begin{align*}
    \Pr{ \theta^j_1\in [\delta,\frac{k_1}{n^4}-\delta], \ldots, \theta^j_n\in [\delta,\frac{k_n}{n^4}-\delta]]} &= \Pr{ \theta^j_1\in [\delta,\frac{k_1}{n^4}-\delta]} \cdots \Pr{\theta^j_n\in [\delta,\frac{k_n}{n^4}-\delta]]} \\
    &\geq (1-\frac{4}{n^{9/4}}) \cdots (1-\frac{4}{n^{9/4}})\\
    & = (1-\frac{4}{n^{9/4}})^n\\
    & \geq 1 - \frac{4}{n^{5/4}}.
\end{align*}
So,
\begin{align*}
    \Pr{ \neg\left(\theta^j_1\in [\delta,\frac{k_1}{n^4}-\delta], \ldots, \theta^j_n\in [\delta,\frac{k_n}{n^4}-\delta]\right)} < \frac{4}{n^{5/4}}.
\end{align*}
Moreover, from Claim~\ref{claim:prob matching} we know that $\Pr{E^j}\leq \frac{1}{n^8}$. So, we get
\begin{align*}
    \sum_{j=1}^{r}\quad & \Pr{ \neg\left(\theta^j_1\in [\delta,\frac{k_1}{n^4}-\delta], \ldots,  \theta^j_n\in [\delta,\frac{k_n}{n^4}-\delta]\right) \text{ or } E^j} \\
    & \leq \sum_{j=1}^r (\frac{4}{n^{5/4}} + \frac{1}{n^8}) \leq r(\frac{4}{n^{5/4}}+\frac{1}{n^8}) \leq n(\frac{4}{n^{5/4}}+\frac{1}{n^8}) < 1.
\end{align*}
\end{proof}

Next, we explain how we can get the desired coloring $\{f_i\}_{i\in\{1,\ldots,n\}}$ using the perfect matchings for graphs $G^j_{\eta_1,\ldots,\eta_n}$.

Recall that $V = \{0, 1,\ldots,n^4-1\}$. Any perfect matching of $G^j_{\eta_1,\ldots,\eta_n}$ gives us a string $c^j \in \{1,\ldots,n\}^V$ with the density function equal to $\mathbf{q}^j$ as follows. Let $M^j_{\eta_1,\ldots,\eta_n}$ be a fixed perfect matching of $G^j_{\eta_1,\ldots,\eta_n}$. Let $c^j:V \to \{1,2,\ldots,n\}$ be defined by $c^j(t) = i$ where $i$ is the unique element of $\{1,2,\ldots,n\}$ such that in the matching $M^j_{\eta_1,\ldots,\eta_n}$, $t\in V$ is connected to $(i,\ell)\in W^j$ for some $0\leq \ell < q^j_i n^4$. It is easy to see that
\begin{align}
\label{eq:r-colors density}
    \frac{\abs{(c^j)^{-1}(\{i\})}}{\abs{V}} = q^j_i.
\end{align}

Recall that
\begin{align}
\label{eq:dist combination}
    \mathbf{p} = \sum_{j=1}^r \alpha_j \mathbf{q}^j.
\end{align}
By Lemma~\ref{lem:one-sided-to-two-sided}, there exists a string $u\in \{1,\ldots,r\}^\N$ with
\begin{align}
\label{eq:color density}
    \abs{ \frac{\abs{u^{-1}(\{j\}) \cap [M,M+N) }}{N} - \alpha_j } \leq \frac{2}{N}
\end{align}
for all $M, N\in\N$, and $j=1,2,\ldots,r$.

Now, we can compose these strings $u,c^1,\ldots,c^r$ in the following way: let $s\in \{1,\ldots,n\}^\N$ be defined by $s((a-1) n^4 + t + 1) = c^{u(a)}(t)$ for $a\in\N$ and $t\in V=\{0, 1,\ldots,n^4-1\}$. From \eqref{eq:r-colors density}, \eqref{eq:dist combination}, and \eqref{eq:color density} we get
\begin{align}
\label{eq:uniform density}
    \lim_{N\to\infty} \sup_{M\in\N}
    \abs{ \frac{ \abs{ s^{-1}(\{i\}) \cap [M,M+N) } }{N} - p_i } = 0,
\end{align}
which means $d_s = \mathbf{p}$. Let $\{f_i\}_{i\in\{1,\ldots,n\}}$ be the coloring of $\N$ corresponding to $s$: for $i=1,\ldots,n$, let $b_1<b_2<\cdots$ be an enumeration of $s^{-1}(\{i\})$ and define $f_i\colon\N\to\N$ by $f_i(m) = b_m$ for all $m\in\N$. We obviously have $d(f_i) = p_i$ for $1\leq i\leq n$.

To find an upper bound for $\text{QR}_1(f_i)$, by our comment after Definition~\ref{def:k-q-r}, we need to have upper and lower bounds for $f_i(m+1)-f_i(m)$ for $m\in\N$. Let $f_i(m+1) = b' = (a'-1) n^4 + t' + 1$ and $f_i(m) = b = (a-1) n^4 + t + 1$ with $a,a'\in\N$ and $t,t'\in V = \{0,1,\ldots,n^4-1\}$. Note that we either have $a'=a$ or $a'=a+1$.

In the first case, where $a'=a$, let $j=u(a)$. We have $c^j(t) = i$ and $t'$ is the first $t<t''\in V$ with $c^j(t'') = i$. Recall that we fixed a perfect matching $M^j_{\eta_1,\ldots,\eta_n}$ of $G^j_{\eta_1,\ldots,\eta_n}$. Assume that in $M^j_{\eta_1,\ldots,\eta_n}$, $t\in V$ is matched with $(i,\ell)\in W^j$ and $t'\in V$ is matched with $(i,\ell')\in W^j$. Since $\eta_i\in[\delta,q^j_i-\delta]$, we obviously have $\ell'=\ell+1$. So, from \eqref{eq:graph definition} for $\ell$ and $\ell+1$, and by triangle inequality, we have:
$$1/q^j_i-2\delta n^4\leq \abs{t-t'} \leq 1/q^j_i+2\delta n^4,$$
which implies:
\begin{align}
\label{def:gap_case_1}
    1/q^j_i-2\delta n^4 \leq f_i(m+1)-f_i(m) \leq 1/q^j_i+2\delta n^4.
\end{align}

In the second case, where $a'=a+1$, let $j=u(a)$ and $j'=u(a')$. Since $\eta_i\in[\delta, q^j_i-\delta]$, we have that $t$ is the last $t''\in V$ with $c^j(t'')=i$ and $t'$ is the first $t''\in V$ with $c^{j'}(t'')=i$. So, from \eqref{eq:graph definition} and by triangle inequality, if $D$ is the distance between $t$ and $t'$ mod $n^4$, we have:
$$1/q^j_i-2\delta n^4 \leq D \leq 1/q^j_i+2\delta n^4.$$
Note that $f_i(m+1) - f_i(m) = a'\ n^4 + t' - a\ n^4 - t = n^4 + t' - t = D$, so we have:
\begin{align}
\label{def:gap_case_2}
    1/q^j_i-2\delta n^4 \leq f_i(m+1)-f_i(m) \leq 1/q^j_i+2\delta n^4.
\end{align}

So, by our comment after Definition ~\ref{def:k-q-r}, and since $q^j_i\leq \frac{2}{n}$, $n>2^9$, and $\delta = \sqrt{\frac{2\log n}{n^7}}$, we have:
\begin{align*}
    \text{QR}_1(f_i)
    & = \frac{\sup_{m\in\N} f_i(m+1)-f_i(m)}{\inf_{m\in\N} f_i(m+1)-f_i(m)}\\
    & \leq \frac{1/q^j_i+2\delta n^4}{1/q^j_i-2\delta n^4} = 1 + \frac{4\delta n^4}{1/q^j_i-2\delta n^4}\\
    & \leq 1 + \frac{4 \sqrt{2 n \log n}}{n/2 - 2 \sqrt{2 n \log n}}\\
    & \leq 1 + 64 \sqrt{\frac{\log n}{n}}
\end{align*}

This completes the proof of Proposition~\ref{prop:big_1}.
\end{proof}

\subsection{Case II} Here we assume that all but one probability in the distribution are very small. Before proving the main proposition for this case, we need to prove the following lemma.

\begin{lemma}
\label{lemma:small-intermediate}
Let $(p_0, p_1, p_2, \dots)$ be a probability distribution, $0 < \eps < 1/10$, and $p_0 > 1-\eps$. Then, there is a coloring $\{f_i\}_{i\in\N\cup\{0\}}$ of $\N$ such that:
\begin{enumerate}
    \item $d(f_i) = p_i$ for $i=0,1,2,\ldots$.
    \item $\text{QR}(f_i) \leq 1+10\eps$ for $i=1,2,\ldots$.
\end{enumerate}
\end{lemma}
\begin{proof}
Similar to the proof of Proposition ~\ref{prop:big_1}, we define a bipartite graph and use Hall's Marriage theorem to show that it has a perfect matching. Let
$$X = \N\times\N = \{(i,n)\ | \ i\in\N, n\in\N\},$$
and $Y = \N$. Define a bipartite graph on $X\cup Y$ by putting an edge between $x = (i,n)\in X$ and $y\in Y$ whenever $\floor{\frac{n-\eps/2}{p_i}}\leq y\leq \ceil{\frac{n+\eps/2}{p_i}}$.
Let $A$ be a finite subset of $Y$. Fix $i\in\N$.
\begin{itemize}
  \item Note that since $\eps < 1/10$, neighbors of vertices in $X_i = \{(i,n)\in X \ | \ n\in\N\}$ are disjoint.
  \item If $x = (i,n) \in X$, we have $\deg(x) = \ceil{\frac{n+\eps/2}{p_i}} - \floor{\frac{n-\eps/2}{p_i}} + 1 \geq \eps/p_i$.
  \item Let $B_i\subseteq X_i$ be the set of vertices that are only connected to vertices in $A$. Then, we have $\abs{B_i} \leq \frac{\abs{A}}{\eps/p_i} = \abs{A} \frac{p_i}{\eps}$.
\end{itemize}
Let $B\subseteq X$ be the set of vertices that are only connected to vertices in $A$. We obviously have $B = \cup_{i\in\N} B_i$, and
$$\abs{B} = \sum_{i\in\N} \abs{B_i} \leq \abs{A} \sum_{i\in\N}\frac{p_i}{\eps} \leq \abs{A}.$$
So, by Hall's Marriage theorem, there is an injective map, $f\colon X \to Y$ such that $f$ sends every element $x\in X$ to a neighboring element in $y\in Y$.   

Fix $i\in\N$. Let $f_i\colon\N\to\N$ be defined by $f_i(n) = f((i,n))$ for $n\in\N$. For $N\in\N$, we have
$$(N-1)p_i - \eps/2 \leq \abs{f_i^{-1}\big([N]\big)} \leq (N+1)p_i + \eps/2.$$
So, $d(f_i) = p_i$.

Let $Y_0 = Y \setminus f(X) = Y \setminus \big(\cup_{i\in\N} f_i(\N)\big)$. By the bound on $\abs{f_i^{-1}\big([N]\big)}$, we get that $Y_0$ is not finite. Let $y_1 < y_2 < \cdots$ be an enumeration of $Y_0$, and define $f_0\colon\N\to\N$ by $f_0(n) = y_n$. Again, by the bound on $\abs{f_i^{-1}\big([N]\big)}$, we get $d(f_0) = p_0$.

To complete the proof, and show that the coloring $\{f_i\}_{i\in\N\cup\{0\}}$ satisfies the conditions in the lemma, we need to show that for $i\in\N$ we have $\text{QR}(f_i) \leq 1 + 10 \eps$. For $i\in\N$, we obviously have 
$$\frac{1-\eps}{p_i} - 2 \leq f_i(n+1) - f_i(n) \leq \frac{1+\eps}{p_i} + 2,$$
which implies $\text{QR}(f_i) \leq \frac{1 + \eps + 2 p_i}{1 - \eps - 2 p_i}$, and since $p_i < \eps < 1/10$, we get $\text{QR}(f_i) \leq \frac{1 + 3\eps}{1- 3\eps} < 1 + 10\eps$.
\end{proof}

Now, we state and prove the main proposition for this case.

\begin{proposition}
\label{prop:big_2}
  Let $(p_0, p_1, p_2, \dots)$ be a probability distribution, $0 < \eps < 1/10$, and $p_0 > 1-\eps$. Then, for each natural number $1<c<\frac{1}{10\eps}$ there is a coloring $\{f_i\}_{i\in\N\cup\{0\}}$ of $\N$ such that:
\begin{enumerate}
    \item $d(f_i) = p_i$ for $i=0,1,2,\ldots$.
    \item $\text{QR}_1(f_i) \leq 1+10c\eps$ for $i=1,2,\ldots$.
    \item $\text{QR}_c(f_0) \leq 1+\frac{2}{c}$
\end{enumerate}
\end{proposition}
\begin{proof}
This follows essentially immediately from the previous proposition and Lemma \ref{lemma:composition}. Given a probability distribution $p$ on $\N\cup\{0\}$, for small $c\in\N$ let $\text{expand}_c(p)$ be the probability distribution on $\N\cup\{0\}$ defined by
$$\text{expand}_c(p)(n) = 
\begin{cases}
  c\ p_n & \text{ if } n > 0,\\
  1 - c \sum_{n>0} p_n & \text{ if } n = 0.
\end{cases}
$$

Let $\{g_i\}_{i\in\N\cup\{0\}}$ be the coloring of $\N$ for the probability distribution $\text{expand}_c(p)$ given by Lemma ~\ref{lemma:small-intermediate}. Using $\{g_i\}_{i\in\N\cup\{0\}}$, we define the coloring $\{f_i\}_{i\in\N\cup\{0\}}$ in the following way.
\begin{itemize}
  \item For $i>0$, let $f_i(n) = c g_i(n)$ for $n\in\N$.
  \item Let $Y_0 = \{ m\in\N \ | \ m \neq 0 \mod c\} \cup \{c f_0(n) \ | \ n\in\N\}$. Let $y_1 < y_2 < \cdots$ be an enumeration of $Y_0$. Define $f_0\colon \N\to\N$ by $f_0(n) = y_n$ for $n\in\N$.
\end{itemize}
In words, we assign color $0$ to any $m\in\N$ with $m\neq 0 \mod c$, and we assign color $i$ to any $c g_i(m)$. It is easy to see that $\{f_i\}_{i\in\N\cup\{0\}}$ is a coloring of $\N$.

Now, we show the conditions in the statement for the coloring $\{f_i\}_{i\in\N\cup\{0\}}$. The first two conditions follow easily. To show the third condition, we need to show that for any $q,r\geq c$ and any $n,m\in\N$, we have
$$\frac{q}{r}\frac{f(n+r)-f(n)}{f(m+q)-f(m)} \leq 1 + \frac{2}{c}.$$
Since every element which is not $0 \mod c$ is colored $0$, for any $p\geq c$ and $k\in\N$, we have $p\leq f(k+p) - f(k) \leq p\frac{c}{c-1}$. So, we have
$$\frac{q}{r}\frac{f(n+r)-f(n)}{f(m+q)-f(m)} \leq \frac{q}{r}\frac{r \frac{c}{c-1}}{q} = \frac{c}{c-1} \leq 1 + \frac{2}{c}.$$
\end{proof}

\subsection{Case III} Here we do not put any assumptions on the probability distribution, but we prove a weaker result than the one in Theorem ~\ref{thm:eps-quasi-regular}.

\begin{proposition}
\label{prop:big_3}
Suppose $(p_1, p_2, \dots)$ is a probability distribution. Then there exists a coloring $\{f_i\}_{i\in\N}$ of $\N$ such that
\begin{enumerate}
    \item $d(f_i) = p_i$ for $i\in\N$, and
    \item for all $k \geq 3$, $\text{QR}_k(f_i) \leq (k+2)/(k-2)$ for $i\in\N$.
\end{enumerate}
\end{proposition}

\begin{proof}
Lemma \ref{lem:one-sided-to-two-sided} applied to $\Sigma=\N$ and $p$, gives us a sequence $s\in\Sigma^\N$ with $d_s = p$ and 
\begin{align}
\label{eq:low_dis}
    \abs{ \frac{ \abs{s^{-1}(\{i\}) \cap [M, M+N)} }{N} - p_i } < \frac{2}{N}
\end{align}
for all $M, N\in\N$, and $i\in\N$. Let $\{f_i\}_{i\in\N}$ be the corresponding coloring to $s$. For each $i\in\N$, we obviously have $d(f_i) = d_s(i)$. To complete the proof, we need to show that for all $i\in\N$ and $k \geq 3$ we have $\text{QR}_k(f_i) \leq (k+2)/(k-2)$. Fix $i\in\N$ and $k\geq 3$. To show the inequality, we need to show that for all $q,r\geq k$ and $m,n\in\N$ we have
$$\frac{q}{r}\frac{f_i(n+r)-f_i(n)}{f_i(m+q)-f_i(m)} \leq \frac{k+2}{k-2},$$
which follows easily from the following claim.
\begin{claim}
For any $n\in\N$ and $r\geq k$ we have
$$\frac{k-2}{k p_i} \leq \frac{f(n+r)-f(n)}{r} \leq \frac{k+2}{k p_i}.$$
\end{claim}
\begin{proof}
Note that
$$\abs{s^{-1}(i)\cap \big[f_i(n), f_i(n+r)\big)} = r.$$
So, by \eqref{eq:low_dis} for $M=f(n)$ and $N=f(n+r)-f(n)$, we get
$$\abs{\frac{r}{f(n+r)-f(n)} - p_i} \leq \frac{2}{f(n+r)-f(n)},$$
which implies
$$\abs{\frac{f(n+r)-f(n)}{r} - \frac{1}{p_i}} \leq \frac{2}{r p_i}.$$
The result follows easily from this, and the fact that $r\geq k$.
\end{proof}
The proof of the claim completes the proof of the proposition.
\end{proof}

\subsection{Proof of Theorem ~\ref{thm:eps-quasi-regular}} 
We have now all the necessary ingredients to prove the theorem. Let $\Sigma$ be a countable alphabet and fix $0 < \eps < 1$. Given the correspondence between colorings of $\N$ and sequences in $\Sigma^\N$, to prove the theorem we just need to show there exists a $\delta>0$ such that whenever $p$ is a probability distribution on $\Sigma$ with $p_\sigma < \delta$ for all $\sigma\in\Sigma$, then there exists a coloring $\{f_\sigma\}_{\sigma\in\Sigma}$ of $\N$ with $d(f_\sigma) = p_\sigma$ and $\text{QR}(f_\sigma) = \text{QR}_1(f_\sigma) \leq 1+\eps$ for all $\sigma\in\Sigma$. We show $\delta = \eps^6 / 10^{20}$ works for us.

Let $\delta = \eps^6 / 10^{20}$ and assume that $p$ is a probability distribution on $\Sigma$ with $p_\sigma < \delta$ for all $\sigma\in\Sigma$. For $i\in\N$ let
$$\Sigma_i = \{\sigma\in\Sigma \ | \ \delta/2^i < p_\sigma \leq \delta/2^{i-1} \}.$$
We call each $\Sigma_i$ a bucket. Let $N = 10^{16} / \eps^4$, $I = \{i\in\N \ | \ \abs{\Sigma_i} > N\}$ and $J = \N \setminus I$. For $L\in\{I,J\}$, let $\Sigma_L = \cup_{l\in L}\Sigma_l$, $p_L = \sum_{\sigma\in\Sigma_L} p_\sigma$. In words, $I$ is the indices of the buckets that have more than $N$ characters, and $J$ is the set of indices of the buckets with less than $N$ characters.

\begin{claim}
\label{claim:g-sigma-coloring}
  There exists a coloring $\{g_\sigma\}_{\sigma\in\Sigma_I}$ of $\N$ such that
  \begin{enumerate}
    \item $d(g_\sigma)$ is proportional to $p_\sigma$ for each $\sigma\in\Sigma_I$,
    \item $\text{QR}_1(g_\sigma) \leq 1 + \frac{1000}{N^{1/4}} = 1 + \frac{\eps}{10}$ for all $\sigma\in\Sigma_I$, and
    \item $\min_{n\in\N} \big( g_\sigma(n+1) - g_\sigma(n) \big) \geq \frac{N}{4}$ for all $\sigma\in\Sigma_I$.
  \end{enumerate}
\end{claim}

First, we use this claim to prove the Theorem, then, we will prove this claim. By Claim~\ref{claim:g-sigma-coloring}, there is a coloring $\{g_\sigma\}_{\sigma\in\Sigma_I}$ of $\N$ with the mentioned properties. Let $\Sigma' = \{ I \} \cup \Sigma_J$ and $\eps' = \eps^2/1000$. Let $p'$ be the probability distribution on $\Sigma'$ defined by $p'_I = p_I$ and $p'_\sigma = p_\sigma$ for each $\sigma\in\Sigma_J$. Note that
$p_J = \sum_{j\in J} \sum_{\sigma\in\Sigma_j} p_\sigma \leq 2N\delta < \eps'$. So, $p'_I = p_I > 1 - \eps'$. Let $c' = \ceil{\frac{10}{\eps}}$ and note that $1 < c' < \frac{1}{10\eps'}$. We can apply Proposition~\ref{prop:big_2} on $p', \eps', c'$ and get a coloring $\{g'_\sigma\}_{\sigma\in\Sigma'}$ such that
\begin{enumerate}
  \item $d(g'_\sigma) = p'_\sigma$ for all $\sigma\in\Sigma'$.
  \item $\text{QR}_1(g'_\sigma) \leq 1 + 10 c' \eps' = 1 + 10\ceil{\frac{10}{\eps}}\frac{\eps^2}{1000} \leq 1 + \eps$ for all $\sigma\in\Sigma_J$.
  \item $\text{QR}_{c'}(g'_I) \leq 1 + \frac{2}{c'} \leq 1 + \frac{\eps}{5}$.
\end{enumerate}

Now, we define the coloring $\{f_\sigma\}_{\sigma\in\Sigma}$ of $\N$. For $\sigma\in\Sigma_J$, let $f_\sigma = g'_\sigma$. So, from the first two properties of the coloring $\{g'_\sigma\}_{\sigma\in\Sigma'}$, we get $d(f_\sigma) = p_\sigma$ and $\text{QR}_1(f_\sigma) \leq 1+\eps$ for all $\sigma\in\Sigma_J$.

For $\sigma\in\Sigma_I$, let $f_\sigma = g'_I \circ g_\sigma$. Fix $\sigma\in\Sigma_I$. By the third property of $g_\sigma$ in Claim ~\ref{claim:g-sigma-coloring} we have 
\begin{align*}
  \min_{n\in\N}\big(g_\sigma(n+1)-g_\sigma(n)\big) \geq N/4 = 10^{16}/(4 \eps^4) \geq c'.
\end{align*}

We can apply Lemma ~\ref{lemma:composition} for $f = g'_I$, $g = g_\sigma$, $k = c'$, and $l = 1$, and get $\text{QR}_1(f_\sigma) \leq (1+\frac{\eps}{5}) (1+\frac{\eps}{10}) \leq 1+\eps$. Moreover, it is straightforward to see that $d(f_\sigma) = p_\sigma$.

So, we showed that for each $\sigma\in\Sigma$, we have $d(f_\sigma) = p_\sigma$ and $\text{QR}_1(f_\sigma) \leq 1 + \eps$. It is easy to show that $\{f_\sigma\}_{\sigma\in\Sigma}$ is a coloring of $\N$. So, the proof of the theorem is complete except from the proof of Claim ~\ref{claim:g-sigma-coloring}, to which we turn now.

\begin{proof}[Proof of Claim ~\ref{claim:g-sigma-coloring}]
For each $i\in I$, let $p^i$ be the normalization of the restriction of $p$ to $\Sigma_i$. So, $p^i$ is a probability distribution on $\Sigma$, and $\frac{\max_{\sigma\in\Sigma_i} p^i_\sigma}{\min_{\sigma\in\Sigma_i} p^i_\sigma} \leq 2$. Hence $p^i_\sigma < \frac{2}{N}$ for all $\sigma\in\Sigma_i$.
By Proposition ~\ref{prop:big_1}, for each $i\in I$, we get a coloring $\{e^{i}_\sigma\}_{\sigma\in\Sigma_i}$ of $\N$ with $d(e^i_\sigma) = p^i_\sigma$ and $\text{QR}_1(e^i_\sigma) < 1 + 64\sqrt{\frac{\log N}{N}}$ for each $\sigma\in\Sigma_i$. Moreover, by Claim ~\ref{claim:min-gap}, for each $i\in I$ and $\sigma\in\Sigma_i$ we have
\begin{align*}
  \min_{n\in\N} e^i_\sigma(n+1) - e^i_\sigma(n) & \geq \frac{1}{p^i_\sigma 
  \big( 1+64\sqrt{\frac{\log N}{N}} \big)} \geq \frac{N}{20}.
\end{align*}

Let $q$ be a probability distribution on $I$ with $q_i$ proportional to $\sum_{\sigma\in\Sigma_i} p_\sigma$ for each $i\in I$. By Proposition ~\ref{prop:big_3}, there is a coloring $\{h_i\}_{i\in I}$ of $\N$ with $d(h_i) = q_i$ and $\text{QR}_k(h_i) \leq (k+2)/(k-2)$ for all $k \geq 3$ and $i\in I$. In particular, for $k = \floor{\frac{N}{20}}$, we get
$$\text{QR}_{\floor{\frac{N}{20}}}(h_i) \leq (\floor{\frac{N}{20}} + 2)/(\floor{\frac{N}{20}} - 2) \leq 1 + \frac{160}{N}.$$

For each $i\in I$ and $\sigma\in\Sigma_i$, let $g_\sigma = h_i \circ e^i_\sigma$. Let $\Sigma_I = \cup_{i\in I} \Sigma_i$. It is easy to see that $\{g_\sigma\}_{\sigma\in\Sigma_I}$ is a coloring of $\N$ and for each $\sigma\in\Sigma_I$, $d(g_\sigma)$ is proportional to $p_\sigma$. Note that for each $i\in I$ and $\sigma\in\Sigma_i$ we can apply Lemma~\ref{lemma:composition} to $f = h_i$, $g = e^i_\sigma$, $k = \floor{\frac{N}{20}}$, and $l = 1$, and get
\begin{align}
\label{eq:qr-1-g-i}
  \text{QR}_1(g_\sigma) \leq \big( 1 + \frac{160}{N} \big) \big( 1 + 64\sqrt{\frac{\log N}{N}} \big) \leq 1 + \frac{1000}{N^{1/4}}.
\end{align}
Moreover, for $\sigma\in\Sigma_I$, since $d(g_\sigma)$ is proportional to $p_\sigma$, we get $d(g_\sigma) < 2/N$, and so by Claim~\ref{claim:min-gap} and ~\eqref{eq:qr-1-g-i}, we get
\begin{align}
\label{eq:min-gap-g-sigma}
  \min_{n\in\N} g_\sigma(n+1) - g_\sigma(n) \geq \frac{N}{2 \big( 1 + \frac{1000}{N^{1/4}} \big) } \geq \frac{N}{4}.
\end{align}

\end{proof}

\bibliography{main}

\begin{bibdiv}
\begin{biblist}

\bib{angel2009discrete}{article}{
      author={Angel, Omer},
      author={Holroyd, Alexander~E},
      author={Martin, James~B},
      author={Propp, James},
       title={Discrete low-discrepancy sequences},
        date={2009},
     journal={arXiv preprint arXiv:0910.1077},
}

\bib{chan1993schedulers}{article}{
      author={Chan, Mee~Yee},
      author={Chin, Francis},
       title={Schedulers for larger classes of pinwheel instances},
        date={1993},
     journal={Algorithmica},
      volume={9},
      number={5},
       pages={425\ndash 462},
}

\bib{fishburn2002pinwheel}{article}{
      author={Fishburn, Peter~C},
      author={Lagarias, Jeffrey~C},
       title={Pinwheel scheduling: Achievable densities},
        date={2002},
     journal={Algorithmica},
      volume={34},
      number={1},
       pages={14\ndash 38},
}

\bib{holte1989pinwheel}{inproceedings}{
      author={Holte, Robert},
      author={Mok, Al},
      author={Rosier, Louis},
      author={Tulchinsky, Igor},
      author={Varvel, Donald},
       title={The pinwheel: A real-time scheduling problem},
organization={IEEE},
        date={1989},
   booktitle={[1989] proceedings of the twenty-second annual hawaii
  international conference on system sciences. volume ii: Software track},
      volume={2},
       pages={693\ndash 702},
}

\bib{kempe2018quasi}{inproceedings}{
      author={Kempe, David},
      author={Schulman, Leonard~J},
      author={Tamuz, Omer},
       title={Quasi-regular sequences and optimal schedules for security
  games},
organization={SIAM},
        date={2018},
   booktitle={Proceedings of the twenty-ninth annual acm-siam symposium on
  discrete algorithms},
       pages={1625\ndash 1644},
}

\bib{lin1997pinwheel}{article}{
      author={Lin, Shun-Shii},
      author={Lin, Kwei-Jay},
       title={A pinwheel scheduler for three distinct numbers with a tight
  schedulability bound},
        date={1997},
     journal={Algorithmica},
      volume={19},
      number={4},
       pages={411\ndash 426},
}

\bib{spencer1982sequences}{article}{
      author={Spencer, Joel},
       title={Sequences with small discrepancy relative to n events},
        date={1982},
     journal={Compositio Mathematica},
      volume={47},
      number={3},
       pages={365\ndash 392},
}

\bib{tijdeman1973distribution}{article}{
      author={Tijdeman, Robert},
       title={On a distribution problem in finite and countable sets},
        date={1973},
     journal={Journal of Combinatorial Theory, Series A},
      volume={15},
      number={2},
       pages={129\ndash 137},
}

\bib{tijdeman1980chairman}{article}{
      author={Tijdeman, Robert},
       title={The chairman assignment problem},
        date={1980},
     journal={Discrete Mathematics},
      volume={32},
      number={3},
       pages={323\ndash 330},
}

\end{biblist}
\end{bibdiv}
\end{document}